\documentclass{article}
\usepackage{arxiv}

\usepackage[utf8]{inputenc} 
\usepackage[T1]{fontenc}    
\usepackage{hyperref}       
\usepackage{amsfonts}       
\usepackage{nicefrac}       
\usepackage{mathtools}
\usepackage{amsthm}
\usepackage{physics}
\usepackage{booktabs} 
\usepackage{wasysym}
\usepackage{bbm}
\usepackage{amssymb}
\usepackage{enumitem}

\theoremstyle{definition}
\newtheorem*{notation}{Notation}
\newtheorem{Definition}{Definition}[section]
\newtheorem{Theorem}[Definition]{Theorem}
\newtheorem{Proposition}[Definition]{Proposition}

\newtheorem{lemma}[Definition]{Lemma}
\newtheorem{remark}[Definition]{Remark}

\newcommand{\R}{\mathbb{R}}
\newcommand{\bH}{\boldsymbol{H}}
\renewcommand{\phi}{\varphi}
\newcommand{\bx}{\boldsymbol{x}}
\newcommand{\bu}{\boldsymbol{u}}
\renewcommand{\dd}{\text{d}}
\newcommand{\bcurl}{\mathbf{curl}}
\newcommand{\bbf}{\mathbf{f}}
\newcommand{\bg}{\mathbf{g}}
\newcommand{\bv}{\boldsymbol{v}}
\newcommand{\bL}{\boldsymbol{{L}}}
\newcommand{\bn}{\mathbf{n}}
\newcommand{\btau}{\boldsymbol{\tau}}

\newcommand{\bphi}{\boldsymbol{\phi}}
\newcommand{\bV}{\mathbf{V}(\Omega)}
\newcommand{\bT}{\mathbf{T}}
\newcommand{\bP}{\mathbf{P}}
\renewcommand{\bT}{\mathbf{T}}
\newcommand{\bI}{\mathbf{I}}
\newcommand{\N}{\mathbb{N}} 
\renewcommand{\P}{\mathbb{P}} 
\renewcommand{\bH}{\boldsymbol{H}}
\newcommand{\bh}{\boldsymbol{h}}
\newcommand{\bk}{\boldsymbol{k}}
\newcommand{\bq}{\boldsymbol{q}}
\newcommand{\bw}{\boldsymbol{w}}
\newcommand{\bz}{\boldsymbol{z}}
\newcommand{\eps}{\varepsilon}
\DeclarePairedDelimiter{\norma}{\lVert}{\rVert}
\newcommand{\bJ}{\mathbf{J}}
\newcommand{\Stot}{\mathcal{S}}
\newcommand{\Sint}{\mathcal{S}^I}
\newcommand{\SB}{\mathcal{S}^{\Gamma}}
\renewcommand{\div}{\text{div}}
\DeclareMathOperator{\ddiv}{div}
\DeclareMathOperator{\ccurl}{curl}
\DeclareMathOperator{\bccurl}{\textbf{curl}}
\renewcommand{\curl}{\text{curl}}

\usepackage{hyperref}
\usepackage{mathtools}
\usepackage{tcolorbox}
\usepackage{amsmath}
\usepackage{nicefrac}

\title{A virtual element method for the solution of 2D time-harmonic elastic wave equations via scalar potentials}

\author{\href{https://orcid.org/0000-0002-4957-3972}{\includegraphics[scale=0.06]{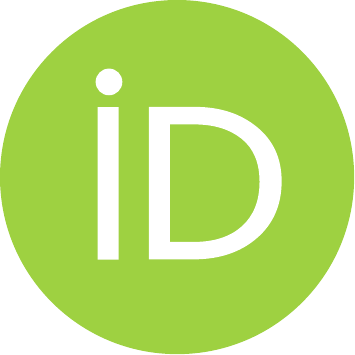}\hspace{1mm}Silvia ~Falletta} \\
	Dipartimento di Scienze Matematiche ``G.L. Lagrange''\\
	Politecnico di Torino\\
	Torino, 10129, Italy \\
	\texttt{silvia.falletta@polito.it} \\
\And
\href{https://orcid.org/0000-0002-2577-1421}{\includegraphics[scale=0.06]{orcid.pdf}\hspace{1mm}Matteo ~Ferrari} \\
	Dipartimento di Scienze Matematiche ``G.L. Lagrange''\\
	Politecnico di Torino\\
	Torino, 10129, Italy \\
	\texttt{matteo.ferrari@polito.it} \\
\And
\href{https://orcid.org/0000-0002-1685-9608}{\includegraphics[scale=0.06]{orcid.pdf}\hspace{1mm}Letizia ~Scuderi} \\
	Dipartimento di Scienze Matematiche ``G.L. Lagrange''\\
	Politecnico di Torino\\
	Torino, 10129, Italy \\
	\texttt{letizia.scuderi@polito.it} \\
}

\renewcommand{\headeright}{}
\renewcommand{\undertitle}{}
\renewcommand{\shorttitle}{}

\hypersetup{
pdftitle={VEM_Ela_helm},
pdfsubject={q-bio.NC, q-bio.QM},
pdfauthor={Falletta S.~Ferrari M.~Scuderi L.},
pdfkeywords={time-harmonic elastic wave equation, virtual element method, Helmholtz-Hodge decomposition, scalar potentials},
}

\begin{document}
\maketitle

\begin{abstract}
In this paper, we propose and analyse a numerical method to solve 2D Dirichlet time-harmonic elastic wave equations. The procedure is based on the decoupling of the elastic vector field into scalar Pressure ($P$-) and Shear ($S$-) waves via a suitable Helmholtz-Hodge decomposition. For the approximation of the two scalar potentials we apply a virtual element method associated with different mesh sizes and degrees of accuracy. We provide for the stability of the method and a convergence error estimate in the $L^2$-norm for the displacement field, in which the contributions to the error associated with the $P$- and $S$- waves are separated. In contrast to standard approaches that are directly applied to the vector formulation, this procedure allows for keeping track of the two different wave numbers, that depend on the $P$- and $S$- speeds of propagation and, therefore, for using a high-order method for the approximation of the wave associated with the higher wave number. Some numerical tests, validating the theoretical results and showing the good performance of the proposed approach, are presented.
\end{abstract}


\section{Introduction}
The numerical modelling of elastic waves propagation problems has undergone, in recent years, an increasing interest in many of the mathematical and engineering areas such as, for example, geophysics, acoustics and seismology.

We consider the Dirichlet vector time harmonic elastic equation defined in a 2D bounded homogeneous medium. For its solution, several numerical methods have been proposed and analysed, among which we mention the traditional finite differences, the finite element method and the more recent Virtual Element Method (VEM).

The aim of this paper is to propose a novel approach based on an Helmholtz-Hodge decomposition of the elastic vector field into two scalar potentials. It consists in reformulating the vector equation into a couple of scalar equations, that describe the propagation of $P$- and $S$-waves, respectively. These two equations are coupled by the Dirichlet boundary condition. This approach, originally proposed in \cite{BurelImperialeJoly2012} to solve interior soft-scattering elastodynamic problems by a finite element method, allows for using approximation spaces with different mesh sizes and degrees of accuracy, related to the $P$- and $S$- speeds of propagation. This turns out to be a great advantage in some applications, such as elastic propagation in soft tissues, in which $P$-waves propagate much faster than $S$- ones, an aspect that displacement-based methods are not able to exploit.
It is worth mentioning that, more recently, the same approach has been applied in \cite{AlbellaImperialJolyRodriguez2018, AlbellaMartinezImperialeJolyRodriguez2021} to traction-free interior elastodynamics problems, in \cite{FallettaMonegatoScuderi2019} to solve exterior soft-scattering ones by means of their space-time BIE representations, and in \cite{FallettaMonegatoScuderi2022} by means of the coupling of boundary and finite element methods. To the best of our knowledge, there are no other papers on such a topic which, according to its interesting properties, is worth to be investigated further. 

The novelty of this paper consists in analyzing the potential approach for elastodynamic interior problems in the frequency domain, providing for stability of the variational formulation. Since in this case we cannot make use of the Fredholm theory, we perform the analysis by proving that the bilinear form associated with the variational formulation is T-coercive. This tool has been introduced in \cite{BonnetBenDhiaCiarletZwolf2010,Ciarlet2012} and it is particularly suitable to treat problems with sign-changing coefficients.

For the numerical solution, we apply a VEM and we provide for a convergence error estimate in the $L^2$-norm for the displacement field. By a careful study, we show that the approximation error of the discrete solution can be split into two contributions, associated with the $P$- and $S$- waves. This aspect allows us to use different mesh grids and approximation orders and to retrieve a high accuracy of the global scheme with a low order VEM for the approximation of the $P$- waves which propagate faster than $S$-ones and, hence, are associated with a smaller wave number. The choice of using VEM relies in the possibility of considering meshes whose elements can be of general shape, and to use local discrete spaces of arbitrarily high order by maintaining the simplicity of implementation independent of it. Another advantage of the proposed procedure that we have exploited, is the possibility of using the same discrete spaces and bilinear forms associated with the VEM for the solution of scalar elliptic problems and, hence, of using the corresponding codes.
We mention that, in literature, VEMs have been already applied to solve elasticity problems: in  \cite{BeiraoBrezziMarini2013} for  compressible and nearly incompressible materials in two dimensions, and in \cite{GainTalischiPaulino2014} for three dimensions; in \cite{BeiraoLovadinaMora2015} with mixed formulation; in \cite{ArtioliBeiraoLovadinaSacco2017, ArtioliBeiraoLovadinaSacco20172} for computational guidelines and in  
\cite{AntoniettiManziniMazzieriMouradVerani2021} for elastodynamics interior problems.
In the latter papers, details on the implementation of the vector version of VEM for elasticity are described, in particular for what concerns the virtual element projection of the strain tensor, which turns out to be the main challenge with respect to the scalar VEM. We remark that in our approach this issue is avoided since the proposed method involves the bilinear forms associated with the scalar Helmholtz equation, and hence only the guidelines to construct the classical VEM matrices (see \cite{BeiraoBrezziMariniRusso2014} and \cite{DesiderioFallettaScuderi2021}) are needed.

The paper is organized as follows: in Section \ref{sec2} we present the model problem for the time-harmonic elastic equation and its reformulation based on the Helmholtz-Hodge decomposition of the vector field. We introduce the variational formulation of the problem and we prove its stability by means of the T-coercivity property of the associated bilinear form. In Section \ref{sec:VEM}, for the approximation of the solution of the new problem, we apply a VEM and we prove the stability and the optimal convergence estimate in the $L^2$-norm for the displacement field. In Section \ref{sec3} we describe the algebraic formulation of the numerical scheme. Finally, in the last section, we present some numerical tests which confirm the theoretical results. Even if the theoretical analysis is provided for the polygonal version of the VEM, to highlight the feasibility of the proposed approach when dealing with curved geometries, in the last example we apply the curvilinear version of the VEM to a problem defined in a curved domain. This allows us to avoid the approximation of the geometry and to retrieve the optimal convergence rate with high approximation orders.

\begin{notation}
In what follows, we will use the bold convention to distinguish vector quantities from scalar ones. Given $\Omega \subset \R^2$ and $s \in \R$, we denote by $H^s(\Omega)$ the standard Sobolev space of order $s$, and by $\bH^s(\Omega) = [H^s(\Omega)]^2.$ 
Similarly, we denote by $(\cdot,\cdot)_{L^2(\Omega)}$ the scalar $L^2$-product in $\Omega$, and by $(\cdot,\cdot)_{\boldsymbol{L}^2(\Omega)}$ the vectorial $L^2$-scalar product therein. Similarly, the corresponding Sobolev norms $\| \cdot \|_{H^1(\Omega)}$ and $\| \cdot \|_{\bH^1(\Omega)}$ are defined as well.
 If $\Gamma$ is a Lipschitz curve, then we use the angled bracket
\begin{equation*}
	\langle \lambda, v \rangle_{\Gamma} = \int_{\Gamma} v(\bx) \lambda(\bx) \, \dd \Gamma_{\bx}
\end{equation*}
to denote the $L^2(\Gamma)$ inner product and its extension as the $H^{-\nicefrac{1}{2}}(\Gamma) \times H^{\nicefrac{1}{2}}(\Gamma)$ duality product.
Denoting by $\bu = \bu(\bx) = [u_1(\bx),u_2(\bx)]^T$ a vector field, and $v=v(\bx)$ a scalar function depending on the space variable $\bx = [x_1,x_2]^T$,  we use the following notations for the differential operators
\begin{equation*}
	\begin{array}{lll}
		& \nabla v =
		\begin{bmatrix}
			\partial_{x_1} v 
			\\ \partial_{x_2} v
		\end{bmatrix},
		\quad
		\bccurl v =
		\begin{bmatrix}
			\partial_{x_2} v
			\\ -\partial_{x_1} v
		\end{bmatrix},
		\quad
		\Delta v = \partial^2_{x_1} v + \partial^2_{x_2} v,
		\\ 
		\\ & \ddiv \bu  = \partial_{x_1} u_1 + \partial_{x_2} u_2,
		\quad \ccurl \bu =  \partial_{x_1} u_2 -\partial_{x_2} u_1,
	\end{array}
\end{equation*}
and
\begin{equation*}
	\Delta \bu  = \begin{bmatrix} \partial^2_{x_1} u_1 + \partial^2_{x_2} u_1, \partial^2_{x_1} u_2 + \partial^2_{x_2} u_2\end{bmatrix}^T.
\end{equation*}
\end{notation}
\section{The model problem} \label{sec2}
We consider the motion of a homogeneous, isotropic elastic solid occupying a bounded star shaped domain $\Omega \subset \R^2$ with Lipschitz boundary $\Gamma$. The displacement governing equation in the frequency domain, with the presence of a body force $\bbf$ and with a prescribed Dirichlet condition $\bg$ on $\Gamma$, can be written as 
\begin{align} \label{pb_iniziale}
	\begin{cases}
		-(\lambda + \mu) \nabla (\ddiv \bu(\bx)) - \mu \Delta \bu(\bx) -\rho \kappa^2 \bu (\bx) = \bbf(\bx) & \bx \in \Omega,
		\\ \bu(\bx) = \bg(\bx) & \bx \in \Gamma,
	\end{cases}
\end{align}
where $\lambda>0$ and $\mu>0$ are the Lam{\' e} constants, $\rho>0$ is the material density and $\kappa > 0$ is the frequency. We assume $\bbf \in \bL^2(\Omega)$ and $\bg \in \bH^{\nicefrac{1}{2}}(\Gamma)$.
To reformulate the vector problem in terms of a couple of scalar potential equations, following \cite{BurelImperialeJoly2012} and observing that
\begin{equation*}
	\Delta \bu = \nabla(\ddiv \bu ) - \bcurl(\ccurl \bu ),
\end{equation*}
we rewrite the partial differential equation in \eqref{pb_iniziale} as
\begin{equation} \label{lammu}
	-(\lambda + 2\mu) \nabla(\ddiv \bu(\bx)) + \mu \bcurl(\ccurl \bu(\bx)) - \rho \kappa^2 \bu(\bx) = \bbf(\bx).
\end{equation}
Hence, we split the displacement field $\bu$ by applying the following Helmholtz-Hodge decomposition
\begin{equation} \label{decgradphi}
	\bu = \nabla \phi^P + \bccurl \phi^S,
\end{equation}
in terms of some unknown scalar potentials $\phi^P, \phi^S \in H^1(\Omega)$. We point out that the decomposition \eqref{decgradphi} is not unique since, for any $\Phi \in H^1(\Omega)$ such that $\nabla \Phi + \bccurl \Phi=0$, the potentials $\widetilde{\phi}^P = \phi^P+\Phi, \widetilde{\phi}^S = \phi^S+\Phi$ satisfy \eqref{decgradphi} as well. 
The existence of a Helmholtz-Hodge decomposition is guaranteed by 
Theorem 3.2 in \cite{GiraultRaviart1986}, that we report here for completeness.
\begin{Theorem} \label{theoremgirault}
Let $\Omega$ be open connected, with Lipschitz boundary $\Gamma$, and denote by $\bn$ the unit normal vector on $\Gamma$ pointing outside $\Omega$. Then, each $\bv \in \bL^2(\Omega)$ can be decomposed in a unique way with a Helmholtz-Hodge decomposition of the type $\bv = \nabla v^P + \bccurl v^S$, with potentials $v^P, v^S \in H^1(\Omega)$ satisfying
\begin{equation} \label{pbdualf}
	\begin{cases}
		\Delta v^P(\bx) = \ddiv \bv(\bx) & \bx \in \Omega,
		\\ (\nabla v^P \cdot \bn)(\bx) = (\bv \cdot \bn)(\bx)  & \bx \in \Gamma,
		\\ \int_{\Gamma} v^P(\bx) \, \dd \Gamma_{\bx} = 0,
	\end{cases}
\end{equation}
and
\begin{equation} \label{pbduald}
	\begin{cases}
		\Delta v^S(\bx) = -\ccurl \bv(\bx) & \bx \in \Omega,
		\\ v^S (\bx) = 0 & \bx \in \Gamma.
	\end{cases}
\end{equation} 
\end{Theorem}
In our case, the characterizations \eqref{pbdualf}-\eqref{pbduald} can not be applied to determine $\phi^P$ and $\phi^S$, the functions $\ddiv \bu $, $\bu \cdot \bn$ and $\ccurl \bu$ being unknown. Therefore, to construct a suitable decomposition, we start by recalling the following properties, for a sufficiently smooth potential $\Phi$:
\begin{equation} \label{propfond}
	\begin{aligned}
			& \div(\nabla \Phi) = - \curl(\bccurl \Phi) = \Delta \Phi, \quad \div(\bccurl \Phi ) =  \curl(\nabla \Phi)  =  0, \quad \quad \text{in $\Omega$},
			\\ & \nabla \Phi \cdot \bn = \bccurl \Phi \cdot \btau, \quad \nabla \Phi \cdot \btau = - \bccurl \Phi \cdot \bn,  \hspace{3.43cm} \text{on $\Gamma$,}
	\end{aligned}
\end{equation}
$\btau$ representing the clockwise oriented tangential directions. By properly applying \eqref{propfond}, equation \eqref{lammu} can be reformulated, in terms of the potentials $\phi^P$ and $\phi^S$, as
\begin{equation} \label{original}
	-(\lambda + 2\mu) \nabla(\Delta \phi^P) - \mu \bcurl(\Delta \phi^S) -\rho \kappa^2 \left(\nabla \phi^P + \bccurl \phi^S\right) = \bbf,
\end{equation}
and the associated Dirichlet boundary conditions can be equivalently rewritten as
\begin{equation*}
	(\nabla \phi^P + \bccurl \phi^S)\cdot \bn = \bg \cdot \bn, \quad (\nabla \phi^P + \bccurl \phi^S)\cdot \btau = \bg \cdot \btau,
\end{equation*}
i.e.,
\begin{equation*}
	\partial_{\bn} \phi^P - \partial_{\btau} \phi^S = \mathbf{g} \cdot \mathbf{n} =: {g}_{\mathbf{n}}, \quad \partial_{\bn} \phi^S + \partial_{\btau} \phi^P = \mathbf{g} \cdot \btau =: {g}_{\boldsymbol{\tau}}.
\end{equation*}
To reformulate \eqref{original} in terms of a couple of scalar PDEs, according to Theorem \ref{theoremgirault} we decompose $\mathbf{f} = \nabla f^P + \bccurl f^S$ with $f^P, f^S \in H^1(\Omega)$, and we easily obtain the following equivalent problem for the potentials (see Proposition 3.5.1 of \cite{Burel2014} for details):
\begin{equation} \label{pb_strong}
	\begin{cases}
		\displaystyle{- \Delta \phi^P(\bx)  - \frac{\rho \kappa^2}{\lambda+2\mu} \phi^P(\bx) = \frac{1}{\lambda+2\mu} f^P(\bx)}, & \bx \in \Omega,
		\\ \displaystyle{- \Delta \phi^S(\bx) - \frac{\rho \kappa^2}{\mu} \phi^S(\bx) = \frac{1}{\mu} f^S(\bx)}, & \bx \in \Omega,
		\\ \partial_{\bn} \phi^P (\bx) - \partial_{\btau} \phi^S(\bx) = {g}_{\mathbf{n}}(\bx), & \bx \in \Gamma,
		\\ \partial_{\bn} \phi^S (\bx) + \partial_{\btau} \phi^P (\bx) = {g}_{\boldsymbol{\tau}}(\bx), & \bx \in \Gamma.
	\end{cases}
\end{equation}
The first two equations of Problem \eqref{pb_strong} are Helmholtz equations associated with the longitudinal and transverse wave numbers 
\begin{equation}\label{eq:wave_numbers}
	\kappa_P^2 = \frac{\rho \kappa^2}{\lambda+2\mu}, \qquad \text{and} \qquad \kappa_S^2 = \frac{\rho \kappa^2}{\mu},
\end{equation}
respectively. Formulation \eqref{pb_strong} is of particular interest in many applications of physics, for example, when the problem source is a $P$-wave or a $S$-wave, and the knowledge of the propagation of the $P$- and $S$-waves generated by this source is required.

To obtain the variational formulation of Problem \eqref{pb_strong} we start by multiplying the first and second equations by the corresponding components of the vector function $\bv = (v^P,v^S) \in \bH^1(\Omega)$, we apply the Green formula and we use the second two equations of \eqref{pb_strong} to rewrite the normal unknown derivatives on $\Gamma$ in terms of the tangential ones. In particular, the weak form reads: find $(\phi^P,\phi^S) \in \bH^1(\Omega)$ such that
\begin{equation} \label{system}
	\begin{cases}
		\left(\nabla \phi^P , \nabla v^P\right)_{\bL^2(\Omega)}  - \left\langle \partial_{\btau} \phi^S, v^P \right\rangle_{\Gamma} & - \kappa_P^2 \left(\phi^P,v^P\right)_{L^2(\Omega)} = \frac{1}{\lambda+2\mu}\left (f^P,v^P\right)_{L^2(\Omega)} +\left \langle g_{\mathbf{n}} , v^P \right\rangle_{\Gamma},
		\\ (\nabla \phi^S ,\nabla v^S)_{\bL^2(\Omega)} + \langle \partial_{\btau} \phi^P, v^S \rangle_{\Gamma} & - \kappa_S^2 (\phi^S,v^S)_{L^2(\Omega)}  = \frac{1}{\mu} (f^S,v^S)_{L^2(\Omega)} + \left\langle {g}_{\btau} , v^S \right\rangle_{\Gamma},
	\end{cases}
\end{equation}
for all $(v^P,v^S) \in \bH^1(\Omega)$.

By introducing the bilinear forms $a:H^1(\Omega) \times H^1(\Omega) \to \R$ and $m : L^2(\Omega) \times L^2(\Omega) \to \R$
\begin{equation}\label{eq:bif}
	a(u,v) = \left(\nabla u , \nabla v\right)_{L^2(\Omega)}, \qquad m(u,v) = \left(u ,v\right)_{L^2(\Omega)},
 \end{equation}
we define $\mathcal{B}, \mathcal{K} : \bH^1(\Omega) \times \bH^1(\Omega) \to \R$
\begin{eqnarray}
	\mathcal{B}(\bphi,\bv) & = & a(\phi^P , v^P) + a(\phi^S , v^S) - \left\langle \partial_{\btau} \phi^S, v^P \right\rangle_{\Gamma} + \langle \partial_{\btau} \phi^P, v^S \rangle_{\Gamma},\label{eq:Bop} 
	\\ \mathcal{K}(\bphi,\bv) & = &  \kappa_P^2 m( \phi^P , v^P) + \kappa_S^2 m( \phi^S , v^S), \label{eq:Kop} 
\end{eqnarray}
and the linear form $\mathcal{L}_{\bbf,\bg} :  \bH^1(\Omega) \to \R$
\begin{equation}\label{eq:Lop}
	\mathcal{L}_{\bbf,\bg}(\bv) = \frac{1}{\lambda+2\mu} \left(f^P,v^P\right)_{L^2(\Omega)} +\left \langle g_{\mathbf{n}} , v^P \right\rangle_{\Gamma} + \frac{1}{\mu} (f^S,v^S)_{L^2(\Omega)} + \left\langle {g}_{\btau} , v^S \right\rangle_{\Gamma}.
\end{equation}
Hence, we rewrite \eqref{system} in the following operator notation: find $\bphi \in \bH^1(\Omega)$ such that
\begin{equation*}
	\mathcal{B}(\bphi,\bv) - \mathcal{K}(\bphi,\bv)  = \mathcal{L}_{\mathbf{f},\mathbf{g}}(\bv), \quad \text{for all~} \bv \in \bH^1(\Omega).
\end{equation*}
By following \cite{Burel2014}, we introduce the Hilbert space 
\begin{equation*}
	\bV = \{\bv = (v^P,v^S) \in \bL^2(\Omega) : \nabla v^P + \bccurl v^S \in \bL^2(\Omega) \},
\end{equation*}
endowed with the norm and semi-norm
\begin{equation*}
	\| \bv \|^2_{\bV} = \| \bv \|^2_{\bL^2(\Omega)} + \left\| \nabla v^P + \bccurl v^S \right\|^2_{\bL^2(\Omega)}, \quad | \bv |^2_{\bV} =  \left\| \nabla v^P + \bccurl v^S \right\|^2_{\bL^2(\Omega)}.
\end{equation*}
Observing that, for $\bv = \left(v^P, v^S\right) \in \bV$, it holds
\begin{equation} \label{proput}
	\nabla v^P + \bccurl v^S =
	\begin{bmatrix}
		\ddiv \bv
		\\ -\ccurl \bv
	\end{bmatrix},
\end{equation}
it is immediate to deduce that $\bV = \bH(\text{div},\Omega) \cap \bH(\text{curl},\Omega)$, where 
\begin{align*}
	\bH(\text{div},\Omega) = \{ \bv \in \bL^2(\Omega) : \ddiv \bv \in L^2(\Omega)\},
	\quad \bH(\text{curl} ,\Omega) =  \{ \bv\in \bL^2(\Omega) : \ccurl \bv \in L^2(\Omega)\}
\end{align*}
and that
\begin{equation*}
	\| \bv \|^2_{\bV} = \| \bv \|^2_{\bL^2(\Omega)} + \| \ddiv \bv \|^2_{L^2(\Omega)} + \| \ccurl \bv \|^2_{L^2(\Omega)}.
\end{equation*}
Following the proof of \cite[Lemma 1]{BurelImperialeJoly2012}, it is easy to show that for all $\bphi = \left(\phi^P, \phi^S\right), \bv = \left(v^P,v^S\right) \in\bH^1(\Omega)$, it holds
\begin{equation} \label{fond_B}
	\mathcal{B}(\bphi,\bv) =  \left(\nabla \phi^P + \bccurl \phi^S,  \nabla v^P + \bccurl v^S\right)_{\bL^2(\Omega)},
\end{equation}
which entails that the bilinear form $\mathcal{B}$ is well-defined also in the less regular space $\bV$. Moreover, combining \eqref{proput} with \eqref{fond_B}, for all $\bphi, \bv \in \bV$ it is possible to rewrite 
\begin{equation*} \label{propB}
	\mathcal{B}(\bphi,\bv) = (\ddiv \bphi,\ddiv \bv)_{L^2(\Omega)} +  (\ccurl \bphi,\ccurl \bv)_{L^2(\Omega)}.
\end{equation*}
Therefore, since for $\bphi \in \bV$ it holds $\mathcal{B}(\bphi,\bphi) = | \bphi |^2_{\bV}$, it appears natural to define the variational formulation of Problem \eqref{pb_strong} as follows: find $\boldsymbol{\varphi} \in \bV$ such that
\begin{equation} \label{weak_Pb}
	\mathcal{A}(\bphi,\bv) = \mathcal{B}(\bphi,\bv) - \mathcal{K}(\bphi,\bv)  = \mathcal{L}_{\mathbf{f},\mathbf{g}}(\bv) \quad \forall \ \bv \in \bV.
\end{equation}
It is worth to point out, since the canonical injection of $\bV$ into $\bL^2(\Omega)$ is not compact (see \cite[Proposition 2.7]{AmroucheBernardiDaugeGirault1998}), we can not assert that the operator $\mathcal{K}$ is compact and, hence, we can not directly apply the well known Fredholm theory. To overcome such issue, in the next section we will introduce an isomorphism $\bT : \bV \to \bV$, that will allow us to prove the stability of the weak formulation \eqref{weak_Pb}. The strategy consists in showing that the new bilinear form $\mathcal{A}(\bphi,\bT \bv)$ can be written as the sum of a coercive bilinear form and of a bilinear form associated with a compact operator. This procedure is the main idea of the so called \textit{T-coercivity}, which has been coined in \cite{BonnetBenDhiaCiarletZwolf2010}, and turned out to be particularly useful for problems such as the Maxwell's equations in the time-harmonic regime.
\subsection{Stability of the weak formulation}
In what follows, the notation $Q_1 \apprle Q_2$ (resp. $Q_1 \apprge Q_2$) means that $Q_1$ is bounded from above (resp. from below) by $c\,Q_2$, where $c$ is a positive constant that may depend on $\kappa_P$ and $\kappa_S$ but, unless explicitly stated, does not depend on any other relevant parameter involved in the definition of $Q_1$ and $Q_2$.

We start by providing for the basic definition of the T-coercivity, as given in  \cite{Ciarlet2012}.
\begin{Definition}
Let $\mathcal{H}$ be a Hilbert space. A continuous bilinear form $b : \mathcal{H} \times \mathcal{H} \to \R$ is \textit{T-coercive} if there exists a linear isomorphism $T : \mathcal{H} \to \mathcal{H}$ such that
\begin{equation*}
	|b(u, Tu)| \ge \| u \|_{\mathcal{H}}^2, \qquad  \forall \ u \in \mathcal{H}.
\end{equation*}
\end{Definition}
In \cite[Theorem 1]{Ciarlet2012}, it has been shown that the T-coercivity is a necessary and sufficient condition to guarantee the well-posedness for a variational problem associated with the bilinear form $b$. 

We remark that the T-coercivity can also be extended to allow for compact perturbations of coercive operators (see \cite{Buffa2005}). This will be the case analysed in this paper. In particular, we will show that there exists an invertible operator $\bT : \bV \to \bV$ such that we can apply the Fredholm theory to the following transformed problem, equivalent to \eqref{weak_Pb}: find $\bphi \in \bV$ such that
\begin{equation}  \label{TAeq}
	\mathcal{A} (\bphi, \bT\bv) = \mathcal{L}_{\bbf,\bg}(\bT\bv) \quad \forall \ \bv \in \mathbf{V}.
\end{equation}
To this aim, proceeding as in \cite{SayasBrownHassel2019}, we introduce a projection $\bP : \bV \to \bV$ such that $\bP : \bV \to \bL^2(\Omega)$ is compact. This is achieved by defining $\bP \bv = \nabla \tilde v^P + \bccurl \tilde v^S$ in terms of the Helmholtz-Hodge decomposition of $\bv \in \bV$, where $\tilde v^P, \tilde v^S \in H^1(\Omega)$ are solutions of the following Dirichlet Poisson problems
\begin{equation*}
	\begin{cases}
		\Delta \tilde v^P(\bx) = \ddiv \bv(\bx), & \bx \in \Omega,
		\\ \tilde v^P(\bx) = 0, & \bx \in \Gamma,
	\end{cases}
	\qquad
	\begin{cases}
		\Delta \tilde v^S(\bx) = -\ccurl \bv(\bx), & \bx \in \Omega,
		\\ \tilde v^S (\bx) = 0, & \bx \in \Gamma.
	\end{cases}
\end{equation*}
By virtue of \eqref{propfond}, the following properties hold
\begin{equation}\label{divcurlP}
	\div(\bP \bv) = \ddiv \bv, \quad \curl(\bP \bv) = \ccurl \bv, \quad \bP^2 = \bP.
\end{equation}
Moreover, due to the standard theory related to Poisson problems, from \eqref{divcurlP} we can easily deduce that
\begin{equation*}
	\| \bP \bv \|_{\bV} \apprle \| \bP \bv \|_{\bL^2(\Omega)} + \| \curl(\bP \bv) \|_{\bL^2(\Omega)} + \| \div(\bP \bv) \|_{\bL^2(\Omega)} \apprle \| \bv \|_{\bV}.
\end{equation*}
Further, from standard regularity results for star-shaped domains (see e.g. \cite{Grisvard1985}), it results that $\tilde v^P, \tilde v^S \in H^{1+\varepsilon}(\Omega)$ with $\varepsilon > 0$. This latter entails $ \bP \bv \in \bH^{\varepsilon}(\Omega)$ and
\begin{equation*}
	\| \bP \bv \|_{\bH^{\varepsilon}(\Omega)} \apprle \| \ddiv \bv \|_{L^2(\Omega)} + \| \ccurl \bv \|_{L^2(\Omega)} = | \bv |_{\bV}.
\end{equation*}
According to \eqref{divcurlP}, the operator $\bT = 2\bP-\bI$ is an isomorphism in $\bV$; indeed, it satisfies $\bT^2=\bI$ and 
\begin{equation}\label{divcurlT}
	\div(\bT\bv) = \ddiv \bv, \qquad \ccurl (\bT\bv) = \ccurl \bv.
\end{equation}
\begin{Proposition}\label{Prop:main_result}
The projection  $\bP : \bV \to \bL^2(\Omega)$ is compact and Problem \eqref{pb_strong} is well-posed, assuming $- \kappa_P^2$ and $-\kappa_S^2$ are not eigenvalues of the associated homogeneous Laplace problem.
\end{Proposition}
\begin{proof}
From the compactness of the standard Sobolev embedding $\bH^{\varepsilon}(\Omega) \hookrightarrow \bL^2(\Omega)$ (see e.g. \cite[Theorem 7.1]{DiNezzaPalatucciValdinoci2012}), and the continuity of $\bP : \bV \to \bH^{\varepsilon}(\Omega)$, it easily follows that $\bP : \bV \to \bL^2(\Omega)$ is compact.

By using the relations \eqref{divcurlT} we write
\begin{align*}
	\mathcal{A} (\bphi, \bT\bv) & = \mathcal{B}(\bphi,\bT\bv)  - \mathcal{K}(\bphi,\bT \bv) = \mathcal{B}(\bphi, \bv) - \mathcal{K}(\bphi, (2\bP-\bI)\bv) 
	\\ & = \mathcal{B}(\bphi, \bv) - \kappa_P^2((\bP\bphi + (\bI-\bP)\bphi)^P, (\bP \bv - (\bI-\bP) \bv)^P)_{L^2(\Omega)}
	\\ & \hspace{1.2cm} - \kappa_S^2((\bP\bphi + (\bI-\bP)\bphi)^S, (\bP \bv - (\bI-\bP) \bv)^S)_{L^2(\Omega)}
	\\ & = \mathcal{B}(\bphi, \bv) + \kappa_P^2((\bP\bphi)^P,(\bP\bv)^P)_{L^2(\Omega)}  + \kappa_P^2(((\bI-\bP)\bphi)^P,((\bI-\bP)\bv)^P)_{L^2(\Omega)}
	\\ & \hspace{1.2cm} + \kappa_S^2((\bP\bphi)^S,(\bP\bv)^S)_{L^2(\Omega)}
 + \kappa_S^2(((\bI-\bP)\bphi)^S,((\bI-\bP)\bv)^S)_{L^2(\Omega)}
	\\ & \hspace{1.2cm} -2\kappa_P^2((\bP\bphi)^P,(\bP\bv)^P)_{L^2(\Omega)} - \kappa_P^2(((\bI-\bP)\bphi)^P,(\bP\bv)^P)_{L^2(\Omega)}
	\\ & \hspace{1.2cm} + \kappa_P^2((\bP\bphi)^P,((\bI-\bP)\bv)^P)_{L^2(\Omega)} -2\kappa_S^2((\bP\bphi)^S,(\bP\bv)^S)_{L^2(\Omega)}
	\\ & \hspace{1.2cm} - \kappa_S^2(((\bI-\bP)\bphi)^S,(\bP\bv)^S)_{L^2(\Omega)}
 + \kappa_S^2((\bP\bphi)^S,((\bI-\bP)\bv)^S)_{L^2(\Omega)}
	\\ &  = \widetilde{\mathcal{B}}(\bphi, \bv) - \widetilde{\mathcal{K}}(\bphi,\bv),
\end{align*}
where we have defined the auxiliary bilinear forms $\widetilde{\mathcal{B}}, \widetilde{\mathcal{K}} : \bV \times \bV \to \R$,
\begin{align*}
	\widetilde{\mathcal{B}}(\bphi, \bv) & =  \mathcal{B}(\bphi, \bv) + \kappa_P^2((\bP\bphi)^P,(\bP\bv)^P)_{L^2(\Omega)} + \kappa_P^2(((\bI-\bP)\bphi)^P,((\bI-\bP)\bv)^P)_{L^2(\Omega)}
	\\ & \hspace{1.2cm} +2\kappa_P^2((\bP\bphi)^S,(\bP\bv)^S)_{L^2(\Omega)} + \kappa_S^2((\bP\bphi)^S,(\bP\bv)^S)_{L^2(\Omega)} 
	\\ & \hspace{1.2cm} + \kappa_S^2(((\bI-\bP)\bphi)^S,((\bI-\bP)\bv)^S)_{L^2(\Omega)} +2\kappa_S^2((\bP\bphi)^P,(\bP\bv)^P)_{L^2(\Omega)}
	\\ & \hspace{1.2cm} - \kappa_P^2(((\bI-\bP)\bphi)^P,(\bP\bv)^P)_{L^2(\Omega)} + \kappa_P^2((\bP\bphi)^P,((\bI-\bP)\bv)^P)_{L^2(\Omega)}
	\\ & \hspace{1.2cm}  - \kappa_S^2(((\bI-\bP)\bphi)^S,(\bP\bv)^S)_{L^2(\Omega)} + \kappa_S^2((\bP\bphi)^S,((\bI-\bP)\bv)^S)_{L^2(\Omega)}
\end{align*}
and
\begin{equation*}
	\widetilde{\mathcal{K}}(\bphi, \bv)  = 2(\kappa_P^2+\kappa_S^2)(\bP\bphi,\bP\bv)_{\bL^2(\Omega)}.
\end{equation*}
To show that $\widetilde{\mathcal{B}}$ is coercive in $\bV$, we write
\begin{align*}
	\widetilde{\mathcal{B}}(\bphi, \bphi) & = \mathcal{B}(\bphi, \bphi) + \kappa_P^2\| (\bP\bphi)^P\|^2_{L^2(\Omega)} + \kappa_P^2\|((\bI-\bP)\bphi)^P\|^2_{L^2(\Omega)} +2\kappa_P^2\|(\bP\bphi)^S\|^2_{L^2(\Omega)}
	\\ & \hspace{1.95cm} + \kappa_S^2\|(\bP\bphi)^S\|^2_{L^2(\Omega)} + \kappa_S^2\|((\bI-\bP)\bphi)^S\|^2_{L^2(\Omega)}  +2\kappa_S^2\|(\bP\bphi)^P\|^2_{L^2(\Omega)}
	\\ & \hspace{1.95cm}  - \kappa_P^2(((\bI-\bP)\bphi)^P,(\bP\bphi)^P)_{L^2(\Omega)} + \kappa_P^2((\bP\bphi)^P,((\bI-\bP)\bphi)^P)_{L^2(\Omega)} 
	\\ & \hspace{1.95cm} - \kappa_S^2(((\bI-\bP)\bphi)^S,(\bP\bphi)^S)_{L^2(\Omega)} + \kappa_S^2((\bP\bphi)^S,((\bI-\bP)\bphi)^S)_{L^2(\Omega)}
	\\ & \ge \mathcal{B}(\bphi, \bphi) + \kappa_P^2\| (\bP\bphi)^P\|^2_{L^2(\Omega)} + \kappa_P^2\|((\bI-\bP)\bphi)^P\|^2_{L^2(\Omega)}   + \kappa_S^2\|(\bP\bphi)^S\|^2_{L^2(\Omega)}
 + \kappa_S^2\|((\bI-\bP)\bphi)^S\|^2_{L^2(\Omega)}. 
\end{align*}
By using the relation  $\| x - y \|^2 + \| x \|^2 \ge \frac{1}{2} \| y \|^2$, the coercivity of $\widetilde{\mathcal{B}}$ follows from the following inequality
\begin{align*}
	\widetilde{\mathcal{B}}(\bphi, \bphi) & \ge \| \ddiv \bphi\|^2_{L^2(\Omega)} + \| \ccurl \bphi\|^2_{L^2(\Omega)} + \frac{\kappa_P^2}{2}\| \phi^P \|^2_{\bL^2(\Omega)} + \frac{\kappa_S^2}{2}\| \phi^S \|^2_{\bL^2(\Omega)} \apprge \| \bphi \|_{\bV}^2.
\end{align*}
From the coercivity of the bilinear form $\widetilde{\mathcal{B}}$, the invertibility of the associated operator $\widetilde{\mathcal{B}} : \bV \to \bV'$ follows. Moreover, by virtue of the compactness of $\bP : \bV \to \bL^2(\Omega)$, we deduce that $\widetilde{\mathcal{K}} : \bV \to \bV'$, associated to the bilinear form $\widetilde{\mathcal{K}}$, is a compact operator (see e.g. \cite[Proposition 15.3]{SayasBrownHassel2019}).

Then $\widetilde{\mathcal{B}} - \widetilde{\mathcal{K}}$ is a Fredholm operator, and so it is invertible if and only if it is injective that is if and only if the solution of \eqref{TAeq} is unique. Since the uniqueness for \eqref{TAeq} is equivalent to the uniqueness for \eqref{pb_strong}, we can then conclude that the dual formulation \eqref{pb_strong} is well posed if and only if $- \kappa_P^2$ and $-\kappa_S^2$ are not eigenvalues for the associated homogeneous Laplace problem.
\end{proof}
\section{Virtual Element Method}\label{sec:VEM}
Aiming at defining approximation spaces for the $P$- and $S$- waves, associated with different meshes and degrees of accuracy, we introduce two sequences of unstructured tessellations $\{\mathcal{T}_{h_\diamond}\}$, $\diamond = P,S$, both representing coverages of the domain $\Omega$. We denote by $E$ the generic element of $\{\mathcal{T}_{h_\diamond}\}$, and by $h_\diamond=\max_{E\in\mathcal{T}_{h_\diamond}} h_E$ the mesh width, $h_E$ being the diamenter of $E$. 
Concerning the properties of the above meshes, we assume there exists a constant $\varrho>0$ such that, for each element $E\in\mathcal{T}_{h_\diamond}$:
\begin{enumerate}[label=(A.\arabic*), ref=A.\arabic*]
	{\setlength\itemindent{5pt} \item\label{A1} $E$ is star-shaped with respect to a ball of radius greater than $\varrho h_{E}$};
	\vspace{-0.15cm}
	{\setlength\itemindent{5pt} \item\label{A2} the length of any edge of $E$ is greater than $\varrho h_{E}$}.
\end{enumerate}
In what follows we briefly describe the main tools
of the VEM, referring the reader to \cite{AhmadAlsaediBrezziMariniRusso2013} and \cite{BeiraoBrezziCangianiManziniMariniRusso2013} for a deeper presentation. Denoting by $\P_{k_\diamond}(E)$, $\diamond = P,S$, the space of polynomials of degree $k_\diamond \in \N$ associated with an element $E \in \mathcal{T}_{h_\diamond} $, we introduce the local polynomial $H^{1}$-projection $\Pi_{k_\diamond}^{\nabla}:H^{1}(E)\rightarrow \P_{k_\diamond}(E)$, defined such that for $w\in H^{1}(E)$:
\begin{equation*} \label{proj_Pi_nabla}
	\begin{cases}
		(\nabla\Pi_{k_\diamond}^{\nabla}w,\nabla q)_{L^2(E)} = (\nabla w,\nabla q)_{L^2(E)} \qquad  \forall\, q\in \P_{k_\diamond}(E),
		\\[6pt] (\Pi_{k_\diamond}^{\nabla} w,1)_{L^2(\partial E)} = (w,1)_{L^2(\partial E)}.
	\end{cases}
\end{equation*}
Moreover, we consider the local polynomial $L^{2}$-projection operator $\Pi_{k_\diamond}^{0}:L^{2}(E)\rightarrow \P_{k_\diamond}(E)$, defined such that for $w\in L^{2}(E)$:
\begin{equation}\label{proj_Pi_0}
	(\Pi_{k_\diamond}^{0}w,q)_{L^2(E)}= (w,q)_{L^2(E)} \qquad \forall\, q\in \P_{k_\diamond}(E).
\end{equation}
The local projectors $\Pi_{k_\diamond}^{\nabla}$ and $\Pi_{k_\diamond}^{0}$ can be extended to the global ones $\Pi_{k_\diamond}^{\nabla}:H^{1}(\Omega)\rightarrow \P_{k_\diamond}(\mathcal{T}_{h_\diamond})$ and $\Pi_{k_\diamond}^{0}:L^{2}(\Omega)\rightarrow \P_{k_\diamond}(\mathcal{T}_{h_\diamond})$ as follows:
\begin{equation*}
	\left(\Pi_{k_\diamond}^{\nabla} w\right)_{|_E}=\Pi_{k_\diamond}^{\nabla}w_{|_E}    \quad \forall\, w\in H^{1}(\Omega), \quad \left(\Pi_{k_\diamond}^{0} w\right)_{|_E}=\Pi_{k_\diamond}^{0}w_{|_{E}}  \quad \forall\, w\in L^{2}(\Omega),
\end{equation*}
$\P_{k_\diamond}(\mathcal{T}_{h_\diamond})$ being the space of piecewise polynomials with respect to the decomposition $\mathcal{T}_{h_\diamond}$ of $\Omega$.

To describe the virtual element space $Q_{h_\diamond}^{k_\diamond}$, we preliminarily consider for each $E\in\mathcal{T}_{h_\diamond}$ the following local finite dimensional \textit{augmented} virtual space $\widetilde{Q}^{k_\diamond}_{h_\diamond}(E)$ and the local \textit{enhanced} virtual space $Q^{k_\diamond}_{h_\diamond}(E)$ defined as follows: 
\begin{align*}\label{local_VEM_space}
	\widetilde{Q}^{k_\diamond}_{h_\diamond}(E) &=\left\{w_{h_\diamond}\in H^{1}(E) \ : \ \Delta w_{h_\diamond}\in \P_{k_\diamond}(E), \ w_{h_\diamond}\,\raisebox{-.5em}{$\vert_{e_{i}}$}\in \P_{k_\diamond}(e_{i}),\, e_i \subset \partial E \right\},
	\\ Q^{k_\diamond}_{h_\diamond}(E) & =\left\{w_{h_\diamond}\in\widetilde{Q}^{k_\diamond}_{h_\diamond}(E) \ : \ \left(\Pi_{k_\diamond}^{\nabla}w_{h_\diamond} - \Pi_{k_\diamond}^{0}w_{h_\diamond}\right) \in \P_{k_\diamond-2}(E)\right\}.
\end{align*}
It has been shown in \cite[Proposition 2]{AhmadAlsaediBrezziMariniRusso2013} that the dimension of $Q^{k_\diamond}_{h_\diamond}(E)$ is 
\begin{equation*}
	\dim(Q^{k_\diamond}_{h_\diamond}(E))= k_\diamond n_E+\frac{k_\diamond(k_\diamond-1)}{2},
\end{equation*}
$n_E$ being the number of edges of $E$. In particular, a generic element $w_{h_\diamond}$ of $Q^{k\diamond}_{h_\diamond}(E)$ is uniquely determined by the following degrees of freedom:
\begin{itemize}
	\item its values at the $n_E$ vertices of $E$;
	\item its values at $k_\diamond-1$ internal points on every edge $e \subset E$;
	\item the $k_\diamond(k_\diamond-1)/2$ moments of $w_{h_\diamond}$ against a scaled polynomial basis of $\mathbb{P}_{k_\diamond-2}(E)$, i.e.,
	\begin{equation*}\label{moments}
		\frac{1}{|E|}\int\limits_{E}w_{h_\diamond}(\mathbf{x})q(\mathbf{x})\,\dd\mathbf{x} \qquad \forall\, q \in \mathbb{P}_{k_\diamond-2}(E)\ \text{with} \ \|q\|_{L^{\infty}(E)} \apprle 1.
	\end{equation*}
\end{itemize}
Choosing an ordering of the degrees of freedom such that these are indexed by $i=1,\ldots,\dim(Q^{k_\diamond}_{h_\diamond}(E))$, we introduce the operator $\text{dof}_{i}:Q^{k_\diamond}_{h_\diamond}(E) \rightarrow \mathbf{R}$, defined as
\begin{equation*}
	\text{dof}_{i}(w_{h_\diamond})= \text{the value of the $i$-th local degree of freedom of}\, w_{h_\diamond}.
\end{equation*}
On the basis of the definition of the local enhanced virtual space, we construct the global one by
\begin{equation*}\label{global_VEM_space}
	Q^{k_\diamond}_{h_\diamond}=\Bigl\{w_{h_\diamond}\in H^{1}(\Omega) \ : \ w_{{h_\diamond}_{\vert_{\tiny E}}}\in Q^{k_\diamond}_{h_\diamond}(E)\quad \forall E\in\mathcal{T}_{h_\diamond}\Bigr\}.
\end{equation*}
Since we shall deal with functions in the product spaces
\begin{equation*}
	\bH^1(\mathcal{T}_{\bh}) = \prod_{E\in\mathcal{T}_{h_P}} H^1(E) \times \prod_{E\in\mathcal{T}_{h_S}} H^1(E) 
\end{equation*}
and
\begin{equation*}
	\mathbf{V}(\mathcal{T}_{\bh}) =  \prod_{E\in\mathcal{T}_{h_P}} Q_{h_P}^{k_P}(E) \times \prod_{E\in\mathcal{T}_{h_S}} Q_{h_S}^{k_S}(E),
\end{equation*}
we introduce, for $v_{h_\diamond}^\diamond\in H^1(\mathcal{T}_{h_\diamond})$ the broken $H^1$-norm
\begin{equation*}
	\| v_{h_\diamond}^\diamond \|_{H^1(\mathcal{T}_{h_\diamond})}^2 = \sum_{E\in \mathcal{T}_{h_\diamond}} \| v_{h_\diamond}^\diamond \|^2_{H^1(E)},
\end{equation*}
and, for $\bv_{\bh} = (v_{h_P}^P,v_{h_S}^S) \in \mathbf{V}(\mathcal{T}_{\bh})$, the broken $\bH^1$- and $\mathbf{V}$-norms
\begin{equation*}
	\| \bv_{\bh} \|_{\bH^1(\mathcal{T}_{\bh})}^2 = \sum_{\diamond=P,S}\| v_{h_\diamond}^\diamond \|_{H^1(\mathcal{T}_{h_\diamond})}^2, \quad \quad \| \bv_{\bh} \|^2_{\mathbf{V}(\mathcal{T}_{\bh})} = \sum_{\widetilde{E}\in \mathcal{T}_{h_P}\cap \mathcal{T}_{h_S}} \| \bv_{\bh} \|^2_{\mathbf{V}(\widetilde{E})}.
\end{equation*}
We remark that examples of functions belonging to the above mentioned spaces are
\begin{equation*}
	\Pi^{\nabla}_{\bk}(\bv) = (\Pi^\nabla_{k_P} v^P, \Pi^\nabla_{k_S} v^S), \quad \Pi^{0}_{\bk}(\bv) = (\Pi^0_{k_P} v^P, \Pi^0_{k_S} v^S), \quad \bv = (v^P,v^S) \in \bH^1(\Omega)
\end{equation*}
with $\Pi^\nabla_{k_\diamond} v^\diamond, \Pi^0_{k_\diamond} v^\diamond \in \mathbb{P}_{k_\diamond}(\mathcal{T}_{h_\diamond})$.

Following \cite{BeiraoBrezziCangianiManziniMariniRusso2013, BeiraoBrezziMariniRusso2014}, we approximate the bilinear form $a(\cdot,\cdot)$ in \eqref{eq:bif} as follows:
\begin{align*}
	a_{h_\diamond}^\diamond(\phi_{h_\diamond}^\diamond,v_{h_\diamond}^\diamond) & = \sum_{E \in \mathcal{T}_{h_\diamond}} (\nabla \Pi_{k_\diamond}^{\nabla} \varphi^{\diamond}_{h_\diamond}, \nabla \Pi_{k_\diamond}^{\nabla} v_{h_\diamond}^\diamond)_{L^2(E)} + S^E\left(\left(\Pi_{k_\diamond}^{\nabla}  -\text{I}\right)\phi_{h_\diamond}^\diamond,\left(\Pi_{k_\diamond}^{\nabla}-\text{I}\right)v_{h_\diamond}^\diamond\right),
\end{align*}
for $\phi_{h_\diamond}^\diamond,v_{h_\diamond}^\diamond \in Q_{h_\diamond}^{k_\diamond}$, where $S^E$ is a suitable stabilization term defined by
\begin{align*}
	S^E(w_{h_\diamond},v_{h_\diamond})=\sum\limits_{j=1}^{\text{dim}\left(Q_{h_\diamond}^{k_\diamond}\right)}\text{dof}_{j}(w_{h_\diamond})\text{dof}_{j}(v_{h_\diamond}).
\end{align*}
As shown in \cite{BeiraoLovadinaRusso2017}, the discrete bilinear form $a^\diamond_{h_\diamond}(\cdot,\cdot)$ satisfies the following properties: 
\begin{align} 
	& k_\diamond \text{-consistency}: \text{for all~} v_{h_\diamond}\in Q^{k_\diamond}_{h_\diamond} \text{~and for all~} q\in \mathbb{P}_{k_\diamond}(\mathcal{T}_{h_\diamond}): \nonumber
	\\ & \hspace{2.5cm} a^\diamond_{h_\diamond}(v_{h_\diamond},q)=a(v_{h_\diamond},q), \label{consistency}
	\\ & H^1\text{-stability}: \text{for all~} v_{h_\diamond}\in Q^{k_\diamond}_{h_\diamond}: \nonumber
	\\ & \hspace{2cm} |v_{h_\diamond}|^2_{H^1(\mathcal{T}_{h_\diamond})} \apprle a^\diamond_{h_\diamond}(v_{h_\diamond},v_{h_\diamond}) \apprle |v_{h_\diamond}|^2_{H^1(\mathcal{T}_{h_\diamond})}. \label{stability}
\end{align}
Similarly, the approximation of the bilinear form $m(\cdot,\cdot)$ in \eqref{eq:bif} reads
\begin{align*}
	m_{h_\diamond}^\diamond(\phi_{h_\diamond}^\diamond,v_{h_\diamond}^\diamond) & = \sum_{E \in \mathcal{T}_{h_\diamond}} (\Pi_{k_\diamond}^{0} \varphi^{\diamond}_{h_\diamond} , \Pi_{k_\diamond}^{0}  v_{h_\diamond}^\diamond)_{L^2(E)} + |E| S^E\left(\left(\Pi_{k_\diamond}^{0}  -\text{I}\right)\phi_{h_\diamond}^\diamond,\left(\Pi_{k_\diamond}^{0}-\text{I}\right)v_{h_\diamond}^\diamond\right).
\end{align*}
The discrete bilinear form $m_{h_\diamond}^\diamond(\cdot,\cdot)$  satisfies:
\begin{align} 
	& k_\diamond \text{-consistency}: \text{for all~} v_{h_\diamond}\in Q^{k_\diamond}_{h_\diamond} \text{~and for all~} q \in \mathbb{P}_{k_\diamond}(\mathcal{T}_{h_\diamond}): \nonumber
	\\ & \hspace{2.5cm}  m^\diamond_{h_\diamond}(v_{h_\diamond},q)=m(v_{h_\diamond},q), \label{m_consistency}
	\\ & L^2\text{-stability}: \text{for all~} v_{h_\diamond}\in Q^{k_\diamond}_{h_\diamond}: \nonumber
	\\ & \hspace{2cm}  \|v_{h_\diamond}\|^2_{L^2(\Omega)} \apprle m^\diamond_{h_\diamond}(v_{h_\diamond},v_{h_\diamond}) \apprle \|v_{h_\diamond}\|^2_{L^2(\Omega)}.
	\label{m_stability}
\end{align}
On the ground of the above discrete setting, we approximate the bilinear forms $\mathcal{B}, \mathcal{K}$ defined in \eqref{eq:Bop} and \eqref{eq:Kop}, respectively, as follows: for $\bphi_{\bh} = \left(\phi_{h_P}^P,\phi_{h_S}^S\right) \in \mathbf{V}^{\bk}_{\bh}$ and $\bv_{\bh} = \left(v_{h_P}^P,v_{h_S}^S\right) \in \mathbf{V}^{\bk}_{\bh}$, with $\mathbf{V}^{\bk}_{\bh} = Q_{h_P}^{k_P} \times Q_{h_S}^{k_S}$
\begin{align*}
	\mathcal{B}_{\bh}\left(\bphi_{\bh},\bv_{\bh}\right) & = a^P_{h_P}(\phi_{h_P}^P,v_{h_P}^P) +a^S_{h_S}\left(\phi_{h_S}^S, v_{h_S}^S\right) - \left\langle \partial_{\btau} \phi_{h_S}^S, v_{h_P}^P \right\rangle_{\Gamma} + \left\langle \partial_{\btau} \phi_{h_P}^P, v_{h_S}^S \right\rangle_{\Gamma},
	\\[6pt] \mathcal{K}_{\bh}(\bphi_{\bh},\bv_{\bh}) & =  \kappa_P^2 m^P_{h_P}(\phi_{h_P}^P , v_{h_P}^P) + \kappa_S^2 m^S_{h_S}\left(\phi_{h_S}^S, v_{h_S}^S\right).
\end{align*}
Finally, we define the approximation of the linear operator $\mathcal{L}_{\bbf,\bg}$ in \eqref{eq:Lop} as
\begin{equation} \label{eq:Lop_disc}
	\begin{aligned}
		\mathcal{L}_{\bbf_{\bh},\bg}(\bv_{\bh}) = & \frac{1}{\lambda+2\mu}\left (f^P, \Pi_{k_P^*}^0 v_{h_P}^P\right)_{L^2(\Omega)} +\left \langle g_{\mathbf{n}} , v_{h_P}^P \right\rangle_{\Gamma} + \frac{1}{\mu} \left(f^S, \Pi_{k_S^*}^0 v_{h_S}^S\right)_{L^2(\Omega)} + \left\langle {g}_{\btau} , v_{h_S}^S \right\rangle_{\Gamma} 
	\end{aligned}
\end{equation}
where $k_\diamond^* = \max\{1,k_\diamond-2\}$ (see \cite{DesiderioFallettaFerrariScuderi20221}).

Hence, the discrete variational formulation of \eqref{weak_Pb} reads: find $\bphi_{\bh} \in \mathbf{V}_{\bh}^{\bk}$ such that
\begin{equation} \label{weak_Pb_discr}
	\mathcal{A}_{\bh}(\bphi_{\bh},\bv_{\bh}) = \mathcal{B}_{\bh}(\bphi_{\bh},\bv_{\bh}) - \mathcal{K}_{\bh}(\bphi_{\bh},\bv_{\bh})  = \mathcal{L}_{\bbf_{\bh},\bg}(\bv_{\bh}) \quad \text{for all~} \bv_{\bh} \in \mathbf{V}_{\bh}^{\bk}.
\end{equation}

It is worth noticing that, in the above formulas, contrary to the terms associated with the interior domain $\Omega$, whose computation needs the use of the projectors $\Pi_{k_\diamond}^\nabla$ and $\Pi_{k_\diamond}^0$, those associated with the boundary $\Gamma$ can be directly computed, the virtual functions and their tangential derivatives being explicitly known on $\Gamma$. 

Once the approximation $\bphi_{\bh}$ of $\bphi$ has been computed, we define the discrete displacement field $\bu_{\bh}$ as follows
\begin{equation}\label{buh2}
	\bu_{\bh} = \nabla \phi_{h_P}^P + \bccurl \phi_{h_S}^S,
\end{equation}
for which the following main result holds.

\begin{Theorem} \label{mainresult}
Let suppose that the solutions $f^P$ and $f^S$ of \eqref{pbdualf} and \eqref{pbduald}, with datum $\bbf$, satisfy $f^P\in H^{s-1}(\Omega)$ and $f^S\in H^{s-1}(\Omega)$, with $s \ge 3$. Let $\bu \in \bH^{s}(\Omega)$ be the solution of Problem \eqref{pb_iniziale} and $\bu_{\bh}$ its approximation obtained by \eqref{buh2}, $\bphi_{\bh} = (\phi_{h_P}^P,\phi_{h_S}^S)$ being the solution of \eqref{weak_Pb_discr}. Then, the following convergence estimate holds
\begin{align*}
	\| \bu - \bu_{\bh} \|_{\bL^2(\Omega)} &\lesssim \bigl(h_P^{\min\{s,k_P\}} + h_S^{\min\{s,k_S\}}\bigr) \| \bu \|_{\bH^{s}(\Omega)} \nonumber
	\\ &  +  \left(h_P + h_S\right)\bigl(h_P^{\min\{s-1,k_P^*-1\}} +h_S^{\min\{s-1,k_S^*-1\}}\bigr) \| \bbf \|_{\bH^{s-2}(\Omega)}. \label{eq:estimate}
\end{align*}
\end{Theorem}

The proof of Theorem \ref{mainresult} is provided in the next section, by collecting intermediate results concerning the stability and the convergence of the proposed method. 
\subsection{Proof of the main result}
\begin{Proposition}\label{prop:immersion}
For all $\bv \in \bH^1(\Omega)$ it holds
\begin{equation} \label{VH1}
	\| \bv \|_{\bV} \apprle \| \bv \|_{\bH^1(\Omega)}.
\end{equation}
\end{Proposition}
\begin{proof}
Following the proof of \cite[Lemma 1]{BurelImperialeJoly2012}, it is easy to show that for $\bv = (v^P,v^S) \in \bH^1(\Omega)$
\begin{equation*}
	\| \bv \|^2_{\bV} = \| \bv \|_{\bH^1(\Omega)}^2 + \langle \partial_{\btau} v^S,v^P \rangle_{\Gamma} - \langle \partial_{\btau} v^P,v^S \rangle_{\Gamma}.
\end{equation*}
Then, using the H{\"o}lder inequality and the trace theorem, we obtain
\begin{align*}
	\| \bv \|^2_{\bV} & \le \| \bv \|_{\bH^1(\Omega)}^2 + \abs{\langle \partial_{\btau} v^S,v^P \rangle_{\Gamma}} +  \abs{\langle \partial_{\btau} v^P,v^S\rangle_{\Gamma}}
	\\ & \apprle \| \bv \|_{\bH^1(\Omega)}^2 + \| \partial_{\btau} v^S \|_{H^{-\nicefrac{1}{2}}(\Gamma)} \|v^P\|_{H^{\nicefrac{1}{2}}(\Gamma)} + \| \partial_{\btau} v^P \|_{H^{-\nicefrac{1}{2}}(\Gamma)} \|v^S \|_{H^{\nicefrac{1}{2}}(\Gamma)}
	\\ & \apprle \| \bv \|_{\bH^1(\Omega)}^2 + \| v^S \|_{H^{\nicefrac{1}{2}}(\Gamma)} \|v^P\|_{H^{\nicefrac{1}{2}}(\Gamma)} + \| v^P \|_{H^{\nicefrac{1}{2}}(\Gamma)} \|v^S \|_{H^{\nicefrac{1}{2}}(\Gamma)}
	\\ & \apprle \| \bv \|_{\bH^1(\Omega)}^2 + \| v^S \|_{H^1(\Omega)} \|v^P\|_{H^1(\Omega)}  \apprle \| \bv \|_{\bH^1(\Omega)}^2,
\end{align*}
where the last inequality follows from the straightforward one $A^2 + B^2 + AB \le \frac{3}{2} (A^2 + B^2)$ for $A, B \in \R$. Therefore, the assertion \eqref{VH1} implies that the immersion $\bH^1(\Omega) \hookrightarrow \bV$ is continuous.
\end{proof}
Similarly, it is possible to prove that  the immersion $\bH^1(\mathcal{T}_{\bh}) \hookrightarrow \mathbf{V}(\mathcal{T}_{\bh})$ is continuous and hence, for $\bv \in \bH^1(\mathcal{T}_{\bh})$, it holds 
\begin{equation} \label{discrVH}
	\| \bv \|_{\mathbf{V}(\mathcal{T}_{\bh})} \apprle \| \bv \|_{\bH^1(\mathcal{T}_{\bh})}.
\end{equation}

From Proposition \ref{prop:immersion}, we deduce the following approximation result in the $\mathbf{V}(\Omega)$-norm.
\begin{lemma}
For all $\bv \in \bH^{s+1}(\Omega)$, with $s > 0$, it holds
\begin{align}\label{interp_VEM}  
	\| \bv-I_{\bh}(\bv) \|_{\mathbf{V}(\Omega)} \apprle \bigl(h_P^{\min{\{s,k_P\}}}+h_S^{\min{\{s,k_S\}}}\bigr)\left\|\bv\right\|_{\bH^{s+1}(\Omega)},
\end{align}
where $I_{\bh}: \bH^{s+1}(\Omega)\to \mathbf{V}_{\bh}^{\bk}$ is the interpolant operator.
\end{lemma}
\begin{proof}
Estimate \eqref{interp_VEM} follows combining \eqref{VH1} with standard interpolation properties of VEM spaces in the $\bH^1(\Omega)$-norm (see Remark 3.8 related to \cite[Theorem 3.7]{BeiraoRussoVacca2019}).
\end{proof}
We collect in the following lemma some classical approximation results for polynomials on star-shaped domains (see, for instance, \cite{CangianiGeorgoulisPryerSutton2017} and \cite{BrennerScott2008}).
\begin{lemma}
For all $v\in H^{s+1}(\Omega)$ it holds:
\begin{align}
	 \left\|v-\Pi_{k_\diamond}^{0} v\right\|_{L^{2}(\Omega)} + h_\diamond \left|v-\Pi_{k_\diamond}^{0} v\right|_{H^{1}(\mathcal{T}_{h_\diamond})} \apprle h_\diamond^{\text{min}\{s+1,k_\diamond+1\}}\left\|v\right\|_{H^{s+1}(\Omega)}, \qquad  s \geq 0.
	\label{interp_property_Pi_0_nabla}
\end{align}
Moreover, for $\bv \in \bV$ it holds (see \cite[Theorem 5.1]{DupontScott1980})
\begin{align} \label{interp_property_Pi_0_DUpont} 
	\left\| \bv-\Pi_{\bk}^{0} \bv\right\|_{\bL^2(\Omega)}\apprle (h_S+h_P) \| \bv \|_{\bV}.
\end{align}
\end{lemma}
Since the operator $\mathcal{A} = \mathcal{B} - \mathcal{K}$ is invertible but not elliptic (and even, not Fredholm), we can not use standard tools to obtain stability results for its discrete counterpart $\mathcal{A}_{\bh}$. Hence, we consider the slightly different operator $\bar{\mathcal{A}}  = \mathcal{B} + \mathcal{K}$, for which we are able to prove auxiliary properties that will allow us to carry out the theoretical analysis for our discrete operator. We start by showing that $\bar{\mathcal{A}}$ is elliptic in $\bV$; indeed, from \eqref{fond_B} it follows that for all $\bphi = (\phi^P,\phi^S) \in \bV$
\begin{align*}
	\bar{\mathcal{A}}(\bphi,\bphi) & = \mathcal{B}(\bphi,\bphi) + \mathcal{K}(\bphi,\bphi) = \| \ddiv \bphi \|_{L^2(\Omega)}^2 + \| \ccurl \bphi \|_{L^2(\Omega)}^2 + \kappa_P^2 \| \phi^P \|_{L^2(\Omega)}^2 + \kappa_S^2 \| \phi^S \|_{L^2(\Omega)}^2
	\\ & \ge | \bphi |^2_{\bV} + \min{\{k_P^2,k_S^2\}} \| \bphi \|^2_{\bL^2(\Omega)} \apprge \| \bphi \|^2_{\bV}.
\end{align*}
Therefore, $\bar{\mathcal{A}}$ being elliptic and continuous  in the $\mathbf{V}$-norm, it is invertible with continuous inverse. To develop further our theoretical analysis, we introduce the following regularity assumption for the operator $\bar{\mathcal{A}}^{-1}$:
\begin{equation*}\label{eq:reg_Abar}
	\bar{\mathcal{A}}^{-1} : \bL^2(\Omega) \to \bH^{1+\varepsilon}(\Omega) \ \ \text{is continuous, for some} \ \ \varepsilon>0,
\end{equation*} 
which is used in the next lemma for the discrete operator $\bar{\mathcal{A}}_{\bh} = \mathcal{B}_{\bh} + \mathcal{K}_{\bh}$. 
\begin{remark}
The previous assumption is related to the modified bilinear form $\bar{\mathcal{A}}$, which is associated to the weak formulation of the time-harmonic Navier problem with the ``good'' sign. Our assumption is motivated by Section 4 of \cite{MelenkSauter2022}, where such a modified operator is defined and analyzed to obtain wavenumber explicit estimates for the time-harmonic Maxwell equations. 
\end{remark}
\begin{lemma} \label{lemma222}
For any $\bq\in\bV$, there exists one and only one $\bq_{\bh} \in \mathbf{V}_{\bh}^{\bk}$ such that
\begin{equation}\label{pbdual}
	\bar{\mathcal{A}_{\bh}}(\bq_{\bh},\bv_{\bh}) = \bar{\mathcal{A}}(\bq,\bv_{\bh}) \quad \text{for all~} \bv_{\bh} \in \mathbf{V}_{\bh}^{\bk}.
\end{equation}
Moreover, for some $\varepsilon >0$, it holds
\begin{align} \label{prop}
\| \bq_{\bh} \|_{\bV} \apprle \| \bq \|_{\bV}, \qquad \| \bq - \bq_{\bh} \|_{\bL^2(\Omega)} \apprle (h_P^{\varepsilon}+h_S^{\varepsilon}) \| \bq \|_{\bV}.
\end{align}
\end{lemma}
\begin{proof}
The proof is similar to those of \cite[Lemma 4.6]{DesiderioFallettaFerrariScuderi20221} and \cite[Theorem 4.1]{DesiderioFallettaFerrariScuderi2022}. In particular, existence and uniqueness of $\bq_{\bh} \in \mathbf{V}_{\bh}^{\bk}$, solution of \eqref{pbdual}, follow from \eqref{stability} and \eqref{m_stability}, which entail the ellipticity and continuity of $\bar{\mathcal{A}}_{\bh}$ in $\bV$ and in $\mathbf{V}_{\bh}^{\bk}$. Moreover, the first of \eqref{prop} holds according to the continuity of the bilinear form $\bar{\mathcal{A}}$ in $\bV$. In order to prove the second of \eqref{prop}, we use a duality argument. We consider $\widetilde{\bw} = \bq_{\bh} - \bq \in \bL^{2}(\Omega)$ and we set $\bw = \bar{\mathcal{A}}^{-1} \widetilde{\bw} \in \bH^{1+\varepsilon}(\Omega)$. Then, for all $\bz \in \bV$, we have
\begin{equation} \label{rel_imp}
	\bar{\mathcal{A}}(\bw,\bz) = \bar{\mathcal{A}}(\bar{\mathcal{A}}^{-1}\widetilde{\bw},\bz) = (\bar{\mathcal{A}}\bar{\mathcal{A}}^{-1}\widetilde{\bw})(\bz) = (\bq_{\bh} - \bq,\bz)_{\bL^2(\Omega)}.
\end{equation}
From the continuity of $\bar{\mathcal{A}}^{-1}$, we obtain
\begin{equation} \label{regu}
	\| \bw \|_{\bH^{1+\varepsilon}(\Omega)} \apprle \| \bq_{\bh} - \bq \|_{\bL^2(\Omega)}.
\end{equation}
Therefore, by choosing $\bz = \bq_{\bh} - \bq$ in \eqref{rel_imp}, we can write
\begin{align*}
	\| \bq_{\bh} - \bq \|_{\bL^2(\Omega)}^2 & = \bar{\mathcal{A}}(\bw,\bq_{\bh}-\bq) = \bar{\mathcal{A}}(\bw - I_{\bh}(\bw),\bq_{\bh}-\bq) + \bar{\mathcal{A}}(I_{\bh}(\bw),\bq_{\bh}) - \bar{\mathcal{A}}(I_{\bh}(\bw),\bq),
\end{align*}
where $I_{\bh} : \bH^{1+\eps}(\Omega) \to \mathbf{V}_{\bh}^{\bk}$ is the interpolation operator.

Since $\bq_{\bh}$ is the solution of \eqref{pbdual}, we estimate the previous identity as follows:
\begin{align}
	\| \bq_{\bh} - \bq \|_{\bL^2(\Omega)}^2 & \le \abs*{\bar{\mathcal{A}}(\bw - I_{\bh}(\bw),\bq_{\bh}-\bq)} + \abs*{\bar{\mathcal{A}}(I_{\bh}(\bw),\bq_{\bh}) - \bar{\mathcal{A}}_{\bh}(I_{\bh}(\bw),\bq_{\bh})} =: I + II. \label{0}
\end{align}
From the continuity of $\bar{\mathcal{A}}$, \eqref{interp_VEM}, the first of \eqref{prop} and \eqref{regu}, we have:
\begin{align} \nonumber
	I & \apprle \| \bw - I_{\bh}(\bw) \|_{\bV} \| \bq_{\bh} - \bq \|_{\bV} \apprle (h_P^{\varepsilon} + h_S^{\varepsilon}) \| \bw \|_{\bH^{1+\eps}(\Omega)} \| \bq \|_{\bV} 
	\\  & \apprle (h_P^{\varepsilon} + h_S^{\varepsilon}) \| \ \bq_{\bh} - \bq \|_{\bL^{2}(\Omega)} \| \bq \|_{\bV}. \label{I}
\end{align}
Moreover, using the continuity of $\bar{\mathcal{A}}$ and $\mathcal{A}$, the consistency properties \eqref{consistency} and  \eqref{m_consistency}, together with estimates \eqref{discrVH}, \eqref{interp_VEM}, \eqref{interp_property_Pi_0_nabla}, \eqref{interp_property_Pi_0_DUpont}, the first of \eqref{prop} and \eqref{regu}, we have: 
\begin{align} 
	II & \le \abs*{\bar{\mathcal{A}}(I_{\bh}(\bw) - \Pi_{\bk}^{0}(\bw),\bq_{\bh})} + \abs*{\bar{\mathcal{A}}(\Pi_{\bk}^{0}(\bw),\bq_{\bh}) - \bar{\mathcal{A}}_{\bh}(\Pi_{\bk}^{0}(\bw),\bq_{\bh})} + \abs*{\bar{\mathcal{A}}_{\bh}(\Pi_{\bk}^{0}(\bw) - I_{\bh}(\bw),\bq_{\bh})} \nonumber
	\\ & \apprle \| I_{\bh}(\bw) - \Pi_{\bk}^0(\bw) \|_{\mathbf{V}(\mathcal{T}_{\bh})} \| \bq_{\bh} \|_{\bV} \nonumber
	\\ & \apprle  \Big( \| I_{\bh}(\bw) - \bw \|_{\bV} + \| \bw - \Pi_{\bk}^0(\bw) \|_{\mathbf{V}(\mathcal{T}_{\bh})} \Big) \| \bq_{\bh} \|_{\bV} \nonumber
	\\ & \apprle (h_P^\eps + h_S^\eps) \| \bw \|_{\bH^{1+\eps}(\Omega)} \| \bq_{\bh} \|_{\bV} \nonumber
	\\ & \apprle (h_P^\eps + h_S^\eps) \| \bq_{\bh} - \bq \|_{\bL^2(\Omega)} \| \bq \|_{\bV}. \label{II}
\end{align}

Finally, we obtain the second of \eqref{prop} combining \eqref{0}, \eqref{I} and \eqref{II}.
\end{proof}
In the following theorem we show the validity of the inf-sup condition for the discrete bilinear form $\mathcal{A}_{\bh}$.
\begin{Theorem}\label{th:infsup}
For $h_P$ and $h_S$ small enough, it holds
\begin{equation*}
	\sup_{\substack{\bq_{\bh} \in \mathbf{V}_{\bh}^{\bk} \\ \bq_{\bh}\ne \boldsymbol{0}}} \frac{\mathcal{A}_{\bh}(\bw_{\bh} , \bq_{\bh})}{\norma{\bq_{\bh}}_{\bV}} \apprge \norma{\bw_{\bh}}_{\bV} \quad \forall \, \bw_{\bh} \in \mathbf{V}_{\bh}^{\bk}.
\end{equation*}
\end{Theorem}
\begin{proof}
Let us consider $\bw_{\bh} \in \mathbf{V}_{\bh}^{\bk}$,  and $\bq \in \bV$ such that $\bq = {\mathcal{A}^*}^{-1} \bJ \bw_{\bh}$, where ${\mathcal{A}^*} : \bV \to \bV'$ is the adjoint of ${\mathcal{A}}$ and $\bJ : \bV \to \bV'$ denotes the canonical continuous map $(\bJ \bw)(\bz) = (\bw,\bz)_{\bV}$. Hence we obtain 
\begin{equation}\label{primavera} 
	\mathcal{A}(\bz,\bq) = (\bw_{\bh},\bz)_{\bV},
\end{equation}
for all $\bz \in \bV$, with $\| \bq \|_{\bV} \apprle \| \bw_{\bh} \|_{\bV}$. Appealing to Lemma \ref{lemma222}, there exists $\bq_{\bh}$ satisfying
\begin{equation*} \label{AAAA}
	\mathcal{B}_{\bh}(\bq_{\bh},\bv_{\bh}) + \mathcal{K}_{\bh}(\bq_{\bh},\bv_{\bh})  = \mathcal{B}(\bq,\bv_{\bh}) + \mathcal{K}(\bq,\bv_{\bh}) \quad \text{for all~} \bv_{\bh} \in \mathbf{V}_{\bh}^{\bk}
\end{equation*}
 such that \eqref{prop} holds for some $\eps>0$. Now,  we write 
\begin{align*}
	\mathcal{A}_{\bh}(\bw_{\bh} , \bq_{\bh}) & = \mathcal{B}_{\bh}(\bw_{\bh} , \bq_{\bh}) + \mathcal{K}_{\bh}(\bw_{\bh} , \bq_{\bh}) - 2\mathcal{K}_{\bh}(\bw_{\bh} , \bq_{\bh}) 
	\\ & = \mathcal{B}(\bw_{\bh} , \bq) + \mathcal{K}(\bw_{\bh} , \bq) - 2\mathcal{K}_{\bh}(\bw_{\bh} , \bq_{\bh}) 
	\\ & = \mathcal{B}(\bw_{\bh} , \bq) - \mathcal{K}(\bw_{\bh} , \bq) - 2\mathcal{K}_{\bh}(\bw_{\bh} , \bq_{\bh}) + 2\mathcal{K}(\bw_{\bh} , \bq)
	\\ & = \mathcal{A}(\bw_{\bh} , \bq) - 2\mathcal{K}_{\bh}(\bw_{\bh} , \bq_{\bh}) + 2\mathcal{K}(\bw_{\bh} , \bq)
	\\ & = \mathcal{A}(\bw_{\bh} , \bq)- 2 \mathcal{K}_{\bh}(\bw_{\bh},\bq_{\bh} - \bq) + 2(\mathcal{K}-\mathcal{K}_{\bh})(\bw_{\bh},\bq)
	\\ & =: I + II + III.
\end{align*}
From \eqref{primavera}, it follows $I = \| \bw_{\bh} \|_{\bV}^2$ and, from the continuity of $\mathcal{K}_{\bh}$ and \eqref{prop}
\begin{equation*}
	II\gtrsim - (h_S^\eps + h_P^\eps) \| \bw_{\bh} \|_{\bV} \|  \bq \|_{\bV}.
\end{equation*}
To estimate $III$, using \eqref{m_consistency}, which implies the polynomial consistency of $\mathcal{K}_{\bh}$, and \eqref{interp_property_Pi_0_DUpont}, we deduce that
\begin{align*}
	III & \apprle \abs*{\mathcal{K}(\bw_{\bh} -\Pi_{\bk}^0 \bw_{\bh},\bq)} + \abs*{ (\mathcal{K}_{\bh}-\mathcal{K})(\Pi_{\bk}^{0} \bw_{\bh},\bq)} +  \abs*{ \mathcal{K}_{\bh}(\bw_{\bh}-\Pi_{\bk}^0 \bw_{\bh},\bq)}
	\\ & \apprle \| \bw_{\bh} - \Pi_{\bk}^0 \bw_{\bh} \|_{\bL^2(\Omega)} \| \bq \|_{\bL^2(\Omega)} \apprle (h_P+h_S) \| \bw_{\bh} \|_{\bV} \| \bq \|_{\bV}.
\end{align*}
Combining the estimates for $I$, $II$ and $III$, we get
\begin{align*}
	\mathcal{A}_{\bh}(\bw_{\bh} , \bq_{\bh}) \apprge (1-h_S^\eps-h_P^\eps-h_S-h_P) \| \bw_{\bh}\|_{\bV} \| \bq \|_{\bV},
\end{align*}
from which, applying the first of \eqref{prop}, we obtain
\begin{align*}
	\frac{\mathcal{A}_{\bh}(\bw_{\bh}, \bq_{\bh})}{\| \bq_{\bh} \|_{\bV}}  \apprge \frac{\mathcal{A}_{\bh}(\bw_{\bh}, \bq_{\bh})}{\| \bq \|_{\bV}} \apprge (1-h_S^\eps-h_P^\eps-h_S-h_P) \| \bw_{\bh}\|_{\bV}.
\end{align*}
Finally, for $h_S$ and $h_P$ small enough, the assertion of the theorem is proved. 
\end{proof}
The following last lemma regards the error associated with the approximation of the source term $\bbf$.
\begin{lemma} \label{result_f}
Let suppose that the solutions $f^P$ and $f^S$ of \eqref{pbdualf} and \eqref{pbduald}, with datum $\bbf$, satisfy $f^P\in H^{s-1}(\Omega)$ and $f^S\in H^{s-1}(\Omega)$, with $s \ge 3$. Then, for all $\bv_{\bh} \in \mathbf{V}_{\bh}^{\bk}$ it holds
\begin{align*}
	& \abs*{\mathcal{L}_{\bbf,\bg}(\bv_{\bh}) - \mathcal{L}_{\bbf_{\bh},\bg}(\bv_{\bh})} \apprle \left(h_P +  h_S\right)\bigl(h_P^{\min\{s-1,k_P^*-1\}} +h_S^{\min\{s-1,k_S^*-1\}}\bigr) \| \bbf \|_{\bH^{s-2}(\Omega)} \| \bv_{\bh} \|_{\bV} , 
\end{align*}
where $k_\diamond^* = \max\{1,k_\diamond-2\}$.
\end{lemma}
\begin{proof}
From definitions \eqref{eq:Lop},  \eqref{eq:Lop_disc} and \eqref{proj_Pi_0}, we can estimate
\begin{align*}
	& \abs*{\mathcal{L}_{\bbf,\bg}(\bv_{\bh}) - \mathcal{L}_{\bbf_{\bh},\bg}(\bv_{\bh})} \apprle \abs{(f^P,v^P_{h_P} - \Pi_{k_P^*}^0 v^P_{h_P})_{L^2(\Omega)}} + \abs{(f^S,v^S_{h_S} - \Pi_{k_S^*}^0 v^S_{h_S})_{L^2(\Omega)}}
	\\ & \hspace{0.5cm} \apprle  \abs{(f^P - \Pi_{k_P^*}^0 f^P,v^P_{h_P} - \Pi_{k_P^*}^0 v^P_{h_P})_{L^2(\Omega)}} + \abs{(f^S-\Pi_{k_S^*}^0 f^S,v^S_{h_S} - \Pi_{k_S^*}^0 v^S_{h_S})_{L^2(\Omega)}}
	\\ & \hspace{0.5cm} \apprle \| f^P - \Pi_{k_P^*}^0 f^P \|_{L^2(\Omega)} \| v^P_{h_P} - \Pi_{k_P^*}^0 v^P_{h_P} \|_{L^2(\Omega)} + \| f^S - \Pi_{k_S^*}^0 f^S \|_{L^2(\Omega)} \| v^S_{h_S} - \Pi_{k_S^*}^0 v^S_{h_S} \|_{L^2(\Omega)}
	\\ & \hspace{0.5cm} \apprle \bigl(h_P^{\min\{s-1,k_P^*-1\}} \| f^P \|_{H^{s-1}(\Omega)} +  h_S^{\min\{s-1,k_S^*-1\}} \| f^S \|_{H^{s-1}(\Omega)}\bigr)\| \bv_{\bh} - \Pi_{\bk^*}^0 \bv_{\bh} \|_{L^2(\Omega)}
	\\ & \hspace{0.5cm} \apprle \left(h_P +  h_S\right)\bigl(h_P^{\min\{s-1,k_P^*-1\}} +h_S^{\min\{s-1,k_S^*-1\}}\bigr) \| \bbf \|_{\bH^{s-2}(\Omega)} \| \bv_{\bh} \|_{\bV}
\end{align*}
where we have denoted by $\bk^* = (\max\{s,k_P^*\},\max\{s,k_S^*\})$.
\end{proof}

From the above preliminaries, we are finally able to prove the main result.
\begin{proof}[Proof of Theorem \ref{mainresult}]
Existence and uniqueness of $\bphi_{\bh}$ follow from the discrete inf-sup condition of Theorem \ref{th:infsup}. Let $I_{\bh}(\bphi) \in \mathbf{V}_{\bh}^{\bk}$ be the interpolant of $\bphi$. By virtue of Theorem \ref{th:infsup} there exists $\bv_{\bh} \in \mathbf{V}_{\bh}^{\bk}$ such that
\begin{equation*}
	\norma{\bphi_{\bh} - I_{\bh}(\bphi)}_{\bV} \apprle \frac{\mathcal{A}_{\bh}(\bphi_{\bh} - I_{\bh}(\bphi),\bv_{\bh})}{\norma{\bv_{\bh}}_{\bV}}.
\end{equation*}
Since $\bphi$ and $\bphi_{\bh}$ are solution of \eqref{weak_Pb} and \eqref{weak_Pb_discr} respectively, we have
\begin{align*}
	& \norma{\bphi_{\bh} - I_{\bh}(\bphi)}_{\bV} \norma{\bv_{\bh}}_{\bV} \apprle \mathcal{A}_{\bh}(\bphi_{\bh} - I_{\bh}(\bphi) , \bv_{\bh})  = \mathcal{A}_{\bh}(\bphi_{\bh}, \bv_{\bh}) - \mathcal{A}_{\bh}(I_{\bh}(\bphi),\bv_{\bh}) 
	\\ & \hspace{1cm} = \mathcal{L}_{\bbf_{\bh},\bg}(\bv_{\bh}) - \mathcal{A}_{\bh}(I_{\bh}(\bphi) , \bv_{\bh}) + \bigl(\mathcal{A}(\bphi,\bv_{\bh}) - \mathcal{L}_{\bbf,\bg}(\bv_{\bh})\bigr)
	\\ & \hspace{1cm} = \bigl(\mathcal{L}_{\bbf_{\bh},\bg}(\bv_{\bh})- \mathcal{L}_{\bbf,\bg}(\bv_{\bh})\bigr)+ \mathcal{A}(\bphi - I_{\bh}(\bphi),\bv_{\bh}) + \bigl(\mathcal{A}(I_{\bh}(\bphi),\bv_{\bh}) - \mathcal{A}_{\bh}(I_{\bh}(\bphi) , \bv_{\bh})\bigr).
\end{align*}
Then, by using Lemma \ref{result_f}, the continuity of $\mathcal{A}$ and the same calculations as in \eqref{II}, we obtain
\begin{align*}
	& \norma{\bphi_{\bh} - I_{\bh}(\bphi)}_{\bV} \norma{\bv_{\bh}}_{\bV}
	\\ & \hspace{0.5cm} \apprle \left(h_P +  h_S\right)\bigl(h_P^{\min\{s-1,k_P^*-1\}} +h_S^{\min\{s-1,k_S^*-1\}}\bigr) \| \bbf \|_{\bH^{s-2}(\Omega)} \| \bv_{\bh} \|_{\bV}
	\\ & \hspace{1cm} + \norma{\bphi - I_{\bh}(\bphi)}_{\bV} \norma{\bv_{\bh}}_{\bV} + \bigl(h_P^{\min\{s,k_P\}} + h_S^{\min\{s,k_S\}}\bigr) \norma{\bphi}_{\bH^{s+1}(\Omega)} \norma{\bv_{\bh}}_{\bV},
\end{align*}
whence the thesis easily follows combining this latter with \eqref{interp_VEM} in the following  estimate
\begin{align*} 
	\norma{\bu - \bu_{\bh}}_{\bL^2(\Omega)} & = |\bphi - \bphi_{\bh}|_{\bV} \le \norma{\bphi -I_{\bh}(\bphi)}_{\bV} + \norma{\bphi_{\bh} - I_{\bh}(\bphi)}_{\bV}.
\end{align*}
\end{proof}
\section{Algebraic Formulation} \label{sec3}
In this section we briefly describe the construction of the final linear system associated with the numerical scheme \eqref{weak_Pb_discr}.
We denote by $\left\{\Phi_{j}^\diamond\right\}_{j\in\Stot_\diamond}$, with $\diamond = P,S$, the basis functions of the discrete VEM spaces $Q^{k_\diamond}_{h_\diamond}$, $\Stot_\diamond$ being the index sets related to the associated degrees of freedom. 

We re-order and split $\Stot_\diamond=\SB_\diamond\cup\Sint_\diamond$, where $\SB_\diamond$ and $\Sint_\diamond$ denote the sets of the indices related to the degrees of freedom lying on $\Gamma$ and in the interior, respectively. We then expand each component of the unknown function $\bphi_{\bh} = (\phi_{h_P}^P,\phi_{h_S}^S) \in Q_{h_P}^{k_P} \times Q_{h_S}^{k_S}$ as
\begin{equation}\label{interpolants}
	\begin{aligned}
		&\phi_{h_\diamond}^\diamond(\mathbf{x})=\sum\limits_{j\in\Stot_\diamond}\phi_{h_\diamond}^{\diamond,j}\Phi^\diamond_{j}(\mathbf{x}) \quad \text{with} \quad \phi_{h_\diamond}^{\diamond,j}=\text{dof}_{j}\left(\phi_{h_\diamond}^{\diamond}\right).
	\end{aligned}
\end{equation}
Hence, using the basis functions of $Q^{k_\diamond}_{h_\diamond}$ to test the discrete counterpart of our model problem, we get
\begin{align*} \label{VEM_discrete_equation}
	& \sum\limits_{j\in\Stot_P}\phi_{h_P}^{P,j}a_{P}(\Phi^P_{j},\Phi^P_{i}) - \kappa_P^2 \sum\limits_{j\in\Stot_P}\phi_{h_P}^{P,j} m_{P}(\Phi^P_{j},\Phi^P_{i})-\sum\limits_{j\in\SB_S}\phi_{h_S}^{S,j}\left\langle \frac{\partial {\Phi^S_{j}}_{|_\Gamma}}{\partial \btau},{\Phi^P_{i}}_{|_\Gamma}\right\rangle_{\Gamma} 
	\\ & \hspace{5.5cm}= \frac{1}{\lambda+2\mu}\mathcal{F}_{h_P}(\Phi^P_i) +\left\langle g_{\mathbf{n}},{\Phi^P_{i}}_{|_\Gamma}\right\rangle_{\Gamma}, \quad i\in\Stot_P
	\\ & \sum\limits_{j\in\Stot_S}\phi_{h_S}^{S,j}a_{S}(\Phi^S_{j},\Phi^S_{i}) - \kappa_S^2 \sum\limits_{j\in\Stot_S}\phi_{h_S}^{S,j}m_{S}(\Phi^S_{j},\Phi^S_{i})+\sum\limits_{j\in\SB_P}\phi_{h_{P}^{P},j}\left\langle \frac{\partial {\Phi^P_{j}}_{|_\Gamma}}{\partial \btau},{\Phi^S_{i}}_{|_\Gamma}\right\rangle_{\Gamma}
	\\ & \hspace{5.5cm}= \frac{1}{\mu}\mathcal{F}_{h_S}(\Phi^S_i) + \left\langle {g}_{\btau},{\Phi^S_{i}}_{|_\Gamma}\right\rangle_{\Gamma}, \quad i\in\Stot_S,
\end{align*}
where we have set  $\mathcal{F}_{h_\diamond}(\Phi^\diamond_i)
 = \left (f^\diamond, \Pi_{k_\diamond^*}^0 \Phi^\diamond_i\right)_{L^2(\Omega)}$ (see formula \eqref{eq:Lop_disc}).
To write the matrix form of the above linear system, we introduce the stiffness matrices $\mathbb{A}^\diamond$, the mass matrices $\mathbb{M}^\diamond$, the matrices $\mathbb{Q}^\diamond$ and $\mathbb{B}^{PS}, \mathbb{B}^{SP}$ whose entries are defined by
\begin{align*}
	&\mathbb{A}^\diamond_{ij}  =  a_\diamond(\Phi_j^\diamond,\Phi_i^\diamond), \hspace{1cm} \mathbb{M}^\diamond_{ij} =  m_\diamond(\Phi^\diamond_{j},\Phi^\diamond_{i}), & i,j\in\Stot_\diamond \\  &\mathbb{Q}^\diamond_{ij}  =  \left\langle {\Phi^\diamond_{j}}_{|_\Gamma} ,{\Phi^\diamond_{i}}_{|_\Gamma}\right\rangle_{\Gamma}, & i\in\Stot_\diamond, j\in\SB_\diamond \\ &\mathbb{B}^{PS}_{ij}  =  \left\langle \frac{\partial {\Phi^S_{j}}_{|_\Gamma}}{\partial \btau},{\Phi^P_{i}}_{|_\Gamma}\right\rangle_{\Gamma}, & i\in\Stot_P, j\in\SB_S\\ &\mathbb{B}^{SP}_{ij}  =  \left\langle \frac{\partial {\Phi^P_{j}}_{|_\Gamma}}{\partial \btau},{\Phi^S_{i}}_{|_\Gamma}\right\rangle_{\Gamma}, & i\in\Stot_S, j\in\SB_P
\end{align*}
and the right hand side vectors
\begin{align*}
	& \bbf^\diamond = \left[\mathcal{F}_{h_\diamond}(\Phi^\diamond_i)\right]_{i\in\Stot_\diamond}, \hspace{1cm} \bg^P_{\bn} = [g_{\bn}^{P,j}]_{j\in\Stot_P}, \hspace{1cm} \bg^S_{\btau} = \left[g_{\btau}^{S,j}\right]_{j\in\Stot_S}.
\end{align*}
In accordance with the splitting of the set of the degrees of freedom, we consider the block partitioned representation of the above matrices and vectors (with obvious meaning of the notation), and we write the linear system as follows:
\begin{align*}
	\begin{bmatrix}
		\mathbb{A}^{P}_{\Gamma \Gamma}-\kappa_P^2 \mathbb{M}^{P}_{\Gamma \Gamma} & \mathbb{A}^P_{\Gamma \text{I}}- \kappa_P^2 \mathbb{M}^{P}_{\Gamma \text{I}} & -\mathbb{B}^{PS} & \mathbb{O}
		\\
		&&&
		\\ \mathbb{A}^P_{\text{I} \Gamma} - \kappa_P^2 \mathbb{M}^{P}_{\text{I} \Gamma} & \mathbb{A}^P_{\text{I} \text{I}} - \kappa_P^2 \mathbb{M}^{P}_{\text{I} \text{I}} & \mathbb{O} & \mathbb{O}
		\\
		&&&
		\\ \mathbb{B}^{SP} & \mathbb{O} & \mathbb{A}^S_{\Gamma  \Gamma} -\kappa_S^2 \mathbb{M}^S_{\Gamma  \Gamma} & \mathbb{A}^S_{\Gamma \text{I}} - \kappa_S^2 \mathbb{M}^S_{\Gamma  \text{I}}
		\\
		&&&
		 \\ \mathbb{O} &  \mathbb{O} & \mathbb{A}^S_{\text{I} \Gamma} - \kappa_S^2 \mathbb{M}^S_{\text{I} \Gamma} & \mathbb{A}^S_{\text{I} \text{I}} - \kappa_S^2 \mathbb{M}^S_{\text{I} \text{I}}
	\end{bmatrix}
	\begin{bmatrix}
		\boldsymbol{\varphi}^{P}_{\Gamma}
		\\
		\\ \boldsymbol{\varphi}^{P}_{\text{I}}
		\\
		\\ \boldsymbol{\varphi}^{S}_{\Gamma}
		\\
		\\ \boldsymbol{\varphi}^{S}_{\text{I}}
	\end{bmatrix} = 
	\begin{bmatrix}
		\frac{1}{\lambda+2\mu} \mathbf{f}^{P}_{\Gamma} + (\mathbb{Q}^P)^T \mathbf{g}^P_{\mathbf{n}}
		\\ 
		\\
		\frac{1}{\lambda+2\mu} \mathbf{f}^{P}_{\text{I}}
		\\ 
		\\
		\frac{1}{\mu} \mathbf{f}^{S}_{\Gamma}	 + (\mathbb{Q}^S)^T \mathbf{g}^S_{\boldsymbol{\tau}}
		\\ 
		\\
		\frac{1}{\mu} \mathbf{f}^{S}_{\text{I}}
	\end{bmatrix}
\end{align*}
in the unknown vectors
$\displaystyle{
\bphi^\diamond = \left[\phi^{\diamond,j}_{h_\diamond}\right]_{j\in\Stot_\diamond}}
$.

It is worth highlighting that the two solutions $\bphi^P$ and $\bphi^S$ are coupled by means of the matrices $\mathbb{B}^{PS}$ and $\mathbb{B}^{SP}$, whose integral entries are defined on non matching boundary meshes. All the other matrices are associated with uncoupled scalar Helmholtz problems defined on different polygonal tessellations of the domain $\Omega$. Hence they are independently computed by means of standard scalar VEM tools (see \cite{BeiraoBrezziMariniRusso2014}).
\section{Numerical results}

In this section, we apply the proposed method to some boundary value problems to show its effectiveness and to validate the convergence estimate of Theorem \ref{mainresult}. We refer to our approach as \textit{scalar} VEM, and we compare it with the classical VEM method applied to the equation \eqref{pb_iniziale}, which we refer to as the \textit{vector} VEM. The corresponding approximate solutions are denoted by $\bu_{\bh}$ ($\bh = (h_P,h_S)$) and $\bu_h$, respectively, recalling that in the \textit{vector} case we can apply the associated VEM defined on a tessellation $\mathcal{T}_h$ with a unique choice of the mesh parameter $h$.

Both approaches have been implemented by in-house MATLAB codes. In particular, for the \textit{scalar} case, the local stiffness matrices $\mathbb{A}^P, \mathbb{A}^S$ and the mass matrices $\mathbb{M}^P, \mathbb{M}^S$ have been constructed by following the guidelines for general elliptic second order problems in \cite{BeiraoBrezziMariniRusso2014}.
For the \textit{vector} case, we have used the library VEMLab \cite{OrtizBernardinAlvarezHitschfeldKahlerRussoSilvaValenzuelaOlateSanzana}, created to solve Poisson and linear elasticity problems and available only for VEM with approximation order equal to 1. In order to solve our problems, we have properly added the contribution due to the presence of the mass term.

For the generation of the partitioning $\mathcal{T}_{h_\diamond}$ of the computational domain $\Omega$, we have used two softwares: Gmsh to construct unstructured conforming meshes consisting of quadrilaterals (see \cite{GeuzaineRemacle2009}), and the Voronoi mesher of PolyMesher (see \cite{TalischiPaulinoPereiraMenezes2012}). In Figure \ref{fig:meshes} we show some representative meshes used in the forthcoming numerical tests. 
\begin{figure}[h]
\centering
\begin{minipage}[b]{0.25\textwidth}
	\includegraphics[width=\textwidth]{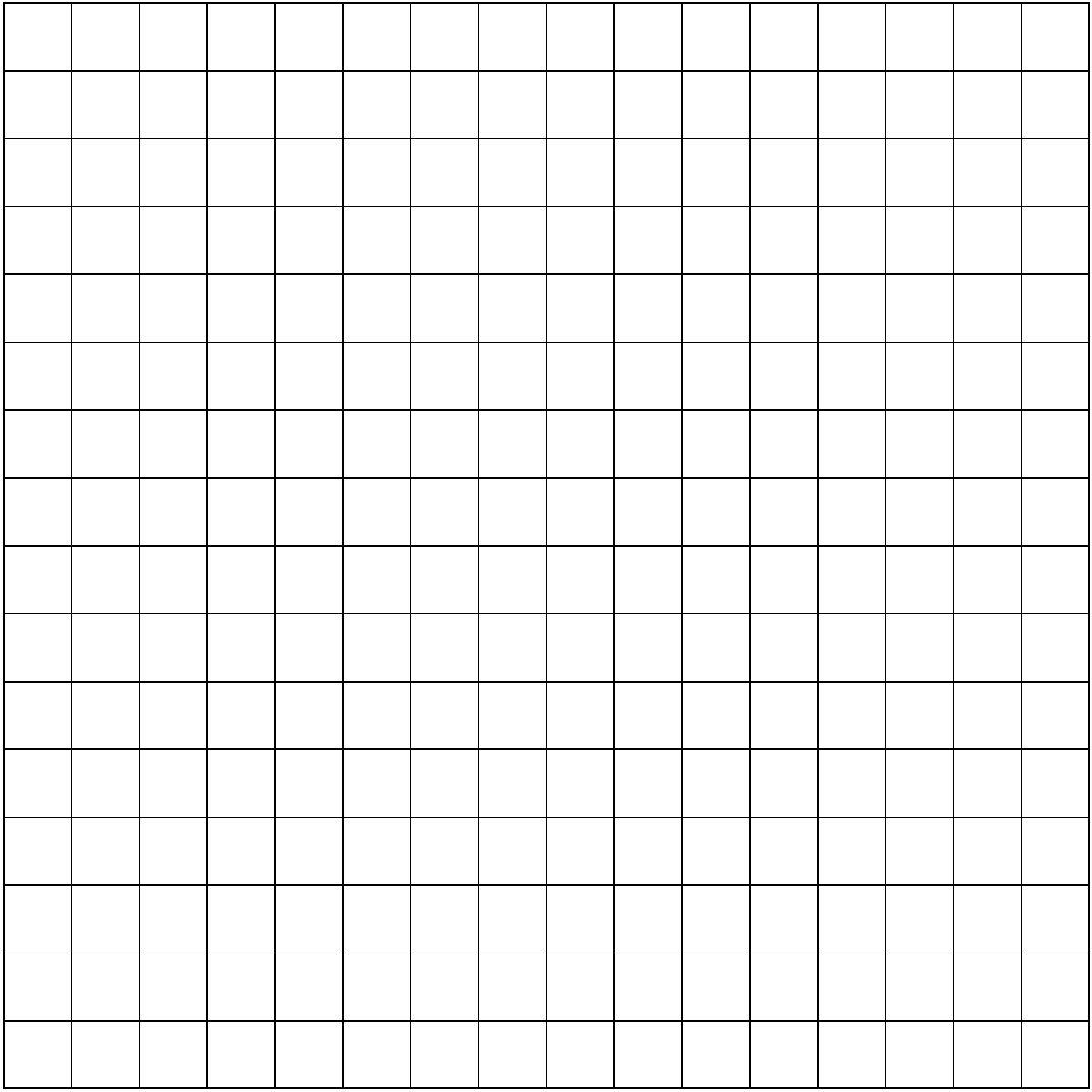}
\end{minipage}
\hspace{0.9cm}
\begin{minipage}[b]{0.25\textwidth}
	\includegraphics[width=\textwidth]{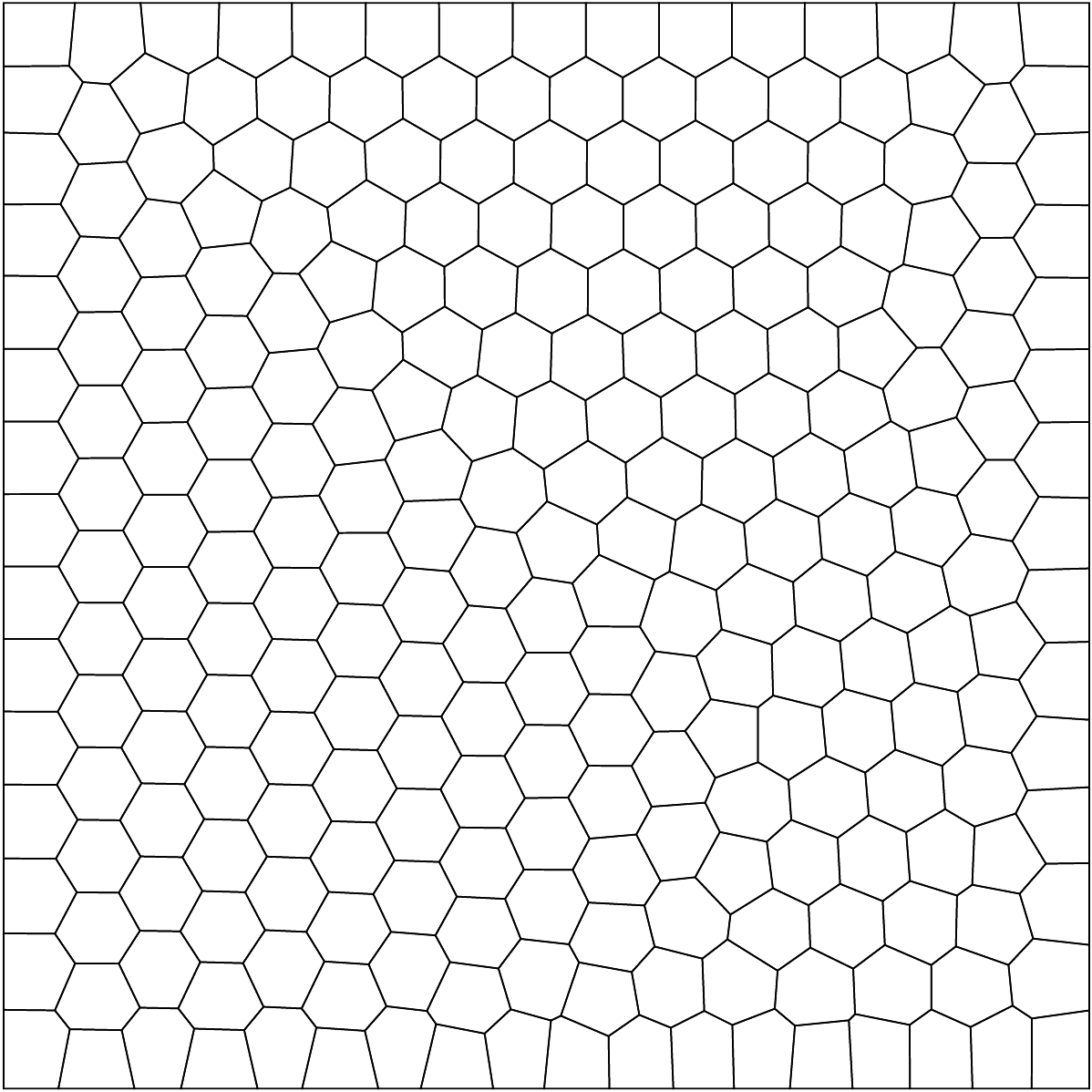}
\end{minipage}
\hspace{0.9cm}
\begin{minipage}[b]{0.25\textwidth}
	\includegraphics[width=\textwidth]{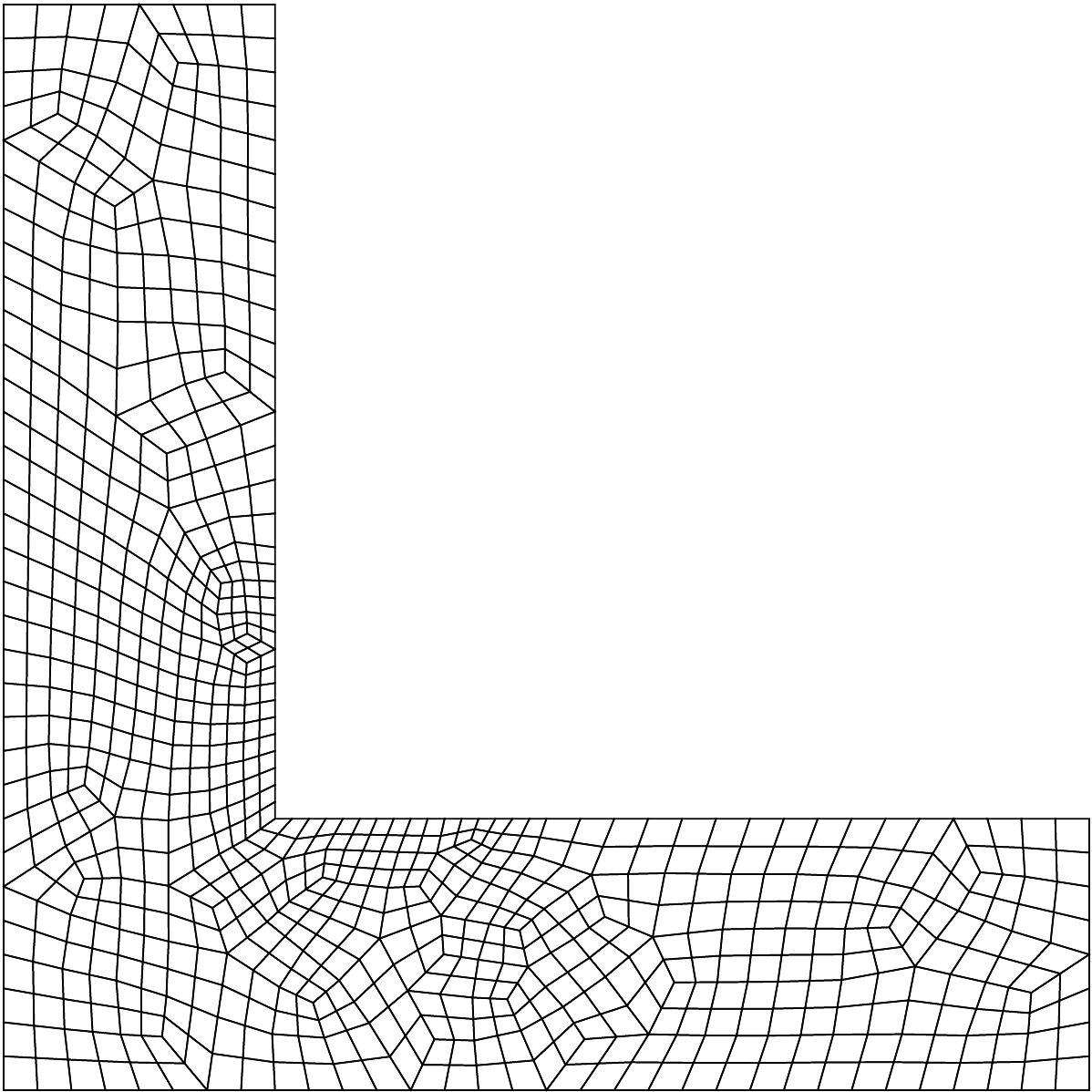}
\end{minipage}
\caption{Representative meshes of the square $[0,1]^2$ (Example 1) obtained by Gmsh (left) and PolyMesher (middle), and of a L-shaped domain (Example 3) obtained by Gmsh (right).}  \label{fig:meshes}
\end{figure}

We point out that, once the approximate solution $\bphi_{\bh} = (\phi_{h_P}^P, \phi_{h_S}^S) \in \mathbf{V}_{\bh}^{\bk}$ of Problem \eqref{weak_Pb_discr} is computed, the displacement field $\bu_{\bh}$ of the original  problem must be reconstructed by means of the relation ${\bu_{\bh}} = 
\nabla \bphi_{\bh} + \bccurl\bphi_{\bh}$,
 which involves the calculation of the partial derivatives of the numerical solutions $\phi_{h_P}^P$ and $\phi_{h_S}^S$ defined in \eqref{interpolants}.
 However, since the analytic expression of these latter is not known, we construct a computable displacement field $\bu_{\bh}^c$ which, in the interior of each element $E$, is defined as follows:
\begin{equation*}\label{buh}
	{\bu_{\bh}^c}_{|_E} =  \nabla \left(\Pi_{k_P}^\nabla{\phi_{h_P}^P}_{|_E}\right) + \bcurl\left({\Pi_{k_S}^\nabla{\phi_{h_S}^S}}_{|_E}\right).
\end{equation*}
As the forthcoming numerical results will show, this formula allows us to retrieve the expected convergence rate of the $L^2$-norm error associated to  the  displacement solution.

It is worth to point out that, to guarantee an  approximation of order $k$ for the displacement field $\bu$, solution of the vector equation \eqref{pb_iniziale},  
it is necessary to approximate the two potentials $(\phi^P,\phi^S)$, solution of the scalar equations \eqref{pb_strong}, by a method of order $k+1$. Despite this aspect, the potential formulation displays the advantage of using different meshes and approximation orders for the pressure and shear potentials, thus allowing to adapt each mesh size to the corresponding wave frequency. For instance, this is crucial when large values of the ratio $\kappa_S/\kappa_P$ are considered since, in this case, it is convenient to use $h_P$ larger than $h_S$ and/or $k_P$ smaller than $k_S$.

\

\noindent \textbf{Example 1.} The purpose of this first test is twofold:  to validate the optimal convergence estimate provided by Theorem \ref{mainresult} and 
to ascertain the so called \textit{patch test}, that is to verify that the method is capable of exactly reproducing polynomial solutions.
To this aim, we deal with the boundary value problem \eqref{pb_iniziale} defined in the unit square $\Omega= (0,1)^2$, with parameters $\lambda=\mu=\rho=1$ and frequency $\kappa=1$. We consider the source term $\mathbf{f} = \nabla f^P + \bccurl f^S$ with
\begin{align*}
	f^P(x_1,x_2) = -x_1-x_2, \quad f^S(x_1,x_2) = -x_2^3 - 6x_2,
\end{align*}
the boundary datum $\bg$ such that the exact solution is 
\begin{align*}
	\bu(x_1,x_2) =  [1+3x_2^2, 1 ]^T
\end{align*}
and the associated scalar potentials are
\begin{align*}
	\phi^P(x_1,x_2) = x_1+x_2, \quad \phi^S(x_1,x_2) = x_2^3.
\end{align*}
We apply the \textit{scalar} and \textit{vector} VEM
 to compute  $\phi^P_{h_P}$, $\phi^S_{h_S}$ and $\bu_h$, approximations of the solutions of Problems \eqref{pb_strong} and \eqref{pb_iniziale}, respectively. The numerical solutions have been obtained by applying the VEM method associated with each of the two aforementioned meshes of the domain $\Omega$. 
For simplicity, we restrict the analysis to the choice $h_P = h_S = h$.

In Tables \ref{tab:prima} and \ref{tab:seconda} we report the number of degrees of freedom (dof) of the \textit{scalar} VEM with respect to the generic order $k$ and to the mesh size $h$ associated with the Gmsh and Voronoi meshes, respectively. As we can see, for a fixed mesh the increase of the dof is approximately linear, while for a fixed order, it is quadratic. Therefore, in terms of computational cost and memory saving, it is more efficient to use a high order VEM rather than fine meshes.
\begin{table}[h!]
	\centering
	\begin{tabular}{lcccccc}
		\toprule
		& $h=0.71$ & $h=0.35$ & $h=0.18$ & $h=0.09$ & $h=0.04$ & $h=0.02$ \\ 
		\toprule
		$k=1$ & $9$	& $25$  & $81$ & $289$ & $1,089$ & $4,225$  \\
		$k=2$ & $21$ & $65$  & $225$ & $833$ & 	$3,201$ & $16,641$ \\
		$k=3$ & $45$ & $153$  & $561$ & $2,145$ & $8,385$ & $33,153$  \\
		\bottomrule
	\end{tabular}
\vskip0.2cm
	\caption{Example 1. Number of the dof for the scalar VEM space of order $k$ for Gmsh mesh size $h$.}
	\label{tab:prima}
\end{table}
\begin{table}[h!]
	\centering
	\begin{tabular}{lcccccc}
		\toprule
		& $h=0.71$ & $h=0.40$ & $h=0.19$ & $h=0.10$ & $h=0.05$ & $h=0.02$ \\ 
		\toprule
		$k=1$ & $9$	& $34$  & $130$ & $514$ & $2,050$ & $8,194$  \\
		$k=2$ & $25$ & $99$  & $387$ & $1,539$ & $6,147$ & $24,579$ \\
		$k=3$ & $45$ & $180$  & $708$ & $2,820$ & $11,268$ & $45,060$  \\
		\bottomrule
	\end{tabular}
\vskip0.2cm
	\caption{Example 1. Number of the dof for the scalar VEM space of order $k$ for Voronoi mesh size $h$.}
	\label{tab:seconda}
\end{table}

In Figure \ref{fig:L2testb} we show the convergence slopes of the $L^2$-norm error, obtained by applying the \textit{scalar} and \textit {vector} VEM, associated with the Gmsh mesh (left plot) and with the Voronoi one (right plot). Some choices of the approximation orders $k_P$ and $k_S$ for the \textit{scalar} approach and $k$ for the \textit{vector} one have been selected. In particular, for the \textit{scalar} VEM we consider decoupled approximation orders $k_P$ and $k_S$, properly chosen to retrieve the expected convergence order. Precisely, we vary $k_P$ and $k_S$ up to order $2$, since for higher orders the exact solution is computed up to the machine precision, thus validating
the \textit{patch-test}. As predicted by the theory and confirmed by the figure, the smallest values which guarantee the optimal convergence order are $k_P = 1$ and $k_S = 2$.
For the \textit{vector} VEM, the expected convergence order is obtained by choosing $k = 1$, while for larger values the polynomial solution is retrieved up to the machine precision. As Figure \ref{fig:L2testb} shows, for both meshes the \textit{scalar} VEM turns out to be slightly more accurate than the vector one.

\begin{figure}[h]
\centering
\begin{minipage}[b]{0.45\textwidth}
	\includegraphics[width=\textwidth]{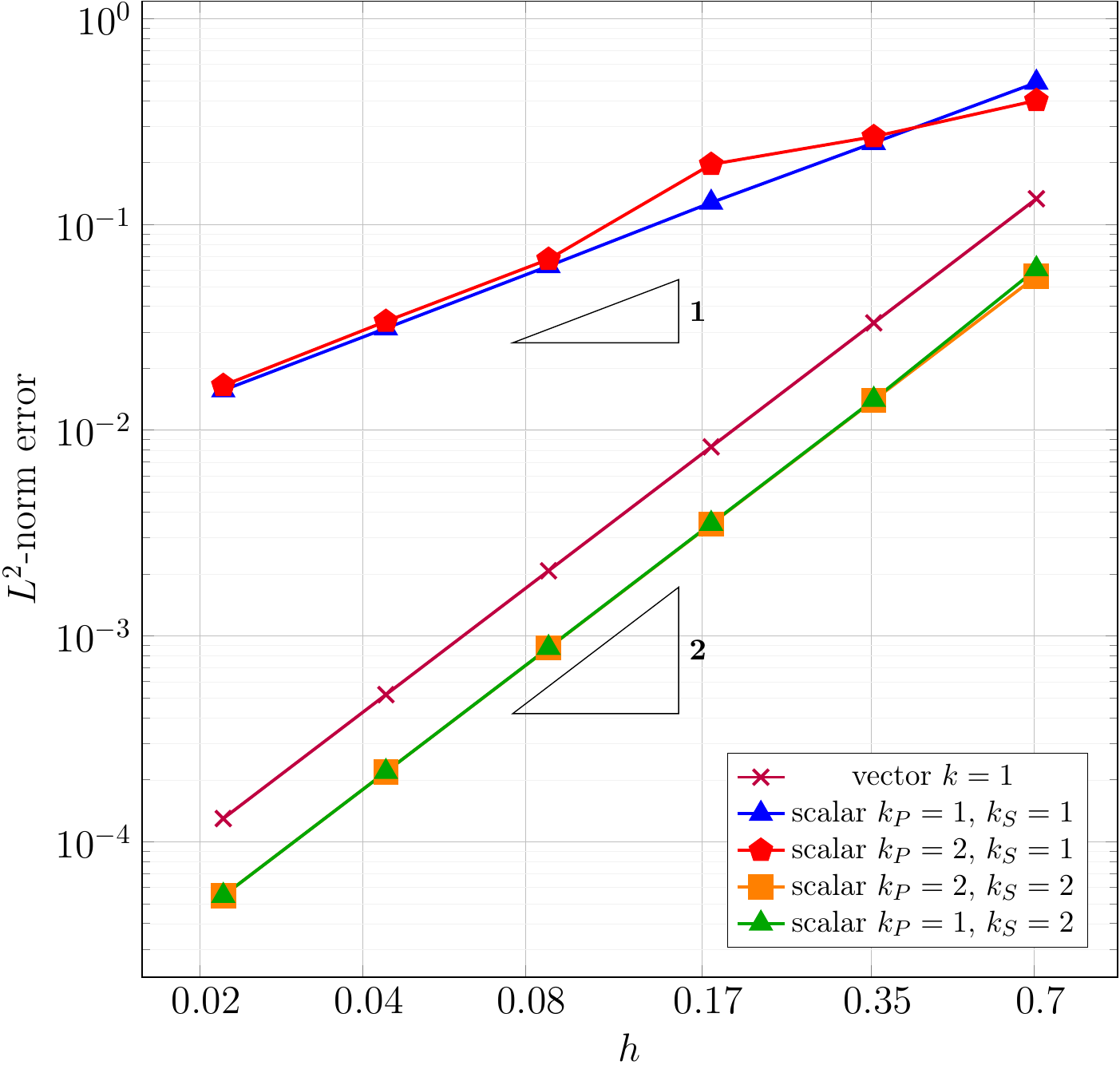}
\end{minipage}
\hfill
\begin{minipage}[b]{0.45\textwidth}
	\includegraphics[width=\textwidth]{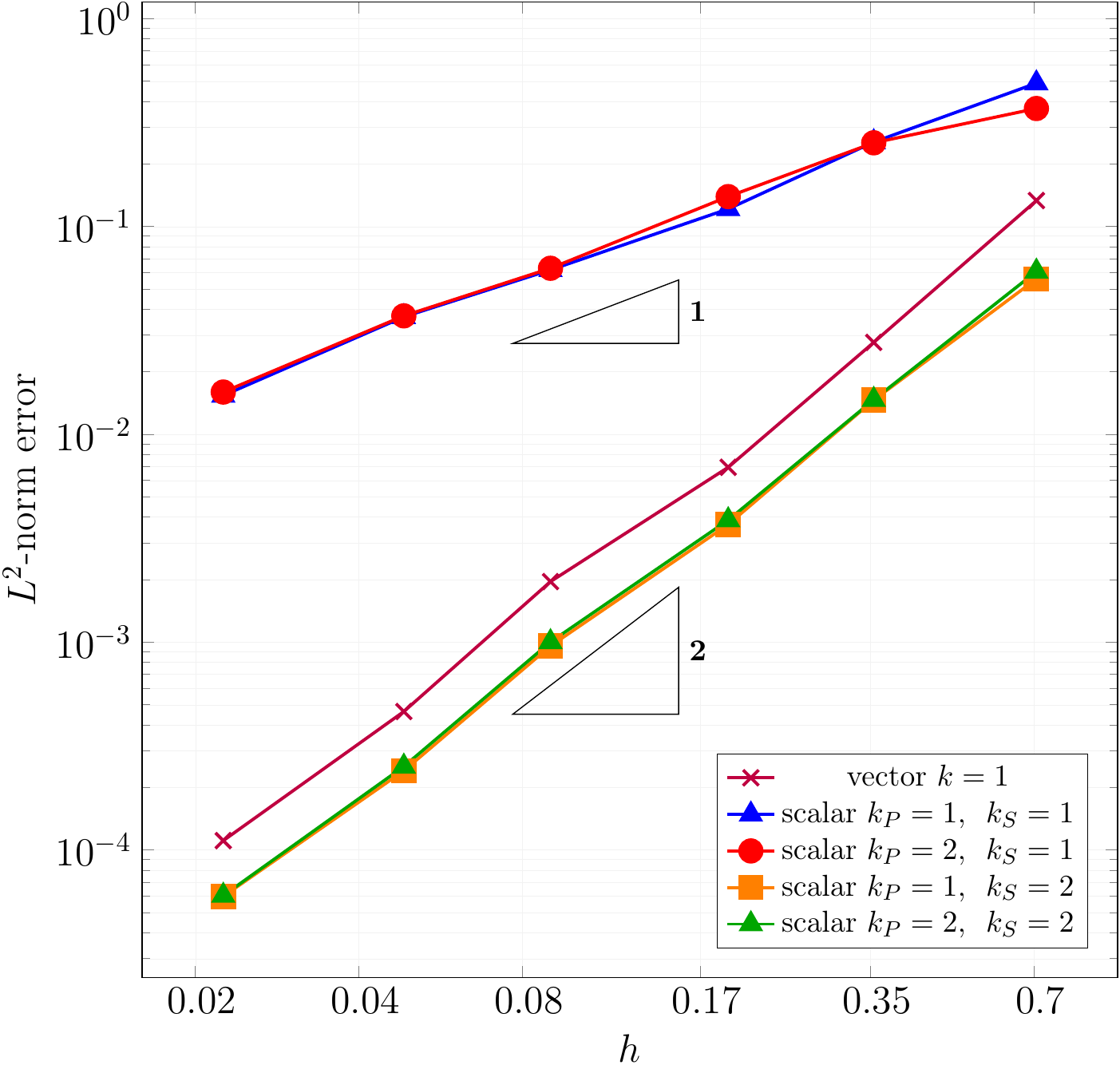}
\end{minipage}
\caption{Example 1. $L^2$-norm absolute errors for the \textit{scalar} and \textit{vector} VEM with respect to the  Gmsh (left) and Voronoi (right) mesh size $h$, by varying $k_P$ and $k_S$.}  \label{fig:L2testb}
\end{figure}

\

\noindent\textbf{Example 2.}  
The purpose of this test is to show the relevance of adapting the mesh sizes and the approximation orders of the \textit{scalar} VEM to the behaviour of the potentials $\phi^P$ and $\phi^S$, which depend on the associated wave-numbers $\kappa_P$ and $\kappa_S$. To this aim we consider Problem \eqref{pb_strong} defined in $\Omega = (0,1)^2$ and with the physical material parameters $\mu = 5.168e$+$08$ N/m$^2$, $\lambda = 1.715e$+$10$  N/m$^2$ and $\rho = 2.320e$+$03$ kg/m$^3$, which correspond to a sandstone layer. 
The frequency is $\kappa = 2.000e\text{+}04$ 1/s; the source terms $f^P$ and $f^S$
and the Dirichlet datum $\mathbf{g}$ are
\begin{align*}
 f^P(x_1,x_2) = \frac{2\cos x_1}{\lambda+2\mu},\quad f^S(x_1,x_2) = 0, \quad \mathbf{g}(x_1,x_2) = \begin{pmatrix}-\sin x_1+\cos x_2 \\ 0 \end{pmatrix}.
\end{align*}
The wave-numbers in the scalar equations are
$\kappa_P = 7.144$ 1/m, $\kappa_S = 4.238e$+$01$ 1/m (see \eqref{eq:wave_numbers}) and, hence, the associated potentials display relevant different behaviours, as shown in Figure \ref{figrefphi}, where their approximations obtained with $k_P=k_S=2$ and $h_P = h_S = 1.1$e-$02$ are plotted.

\begin{figure}[h]
\centering
\begin{minipage}[b]{0.45\textwidth}
	\includegraphics[width=\textwidth]{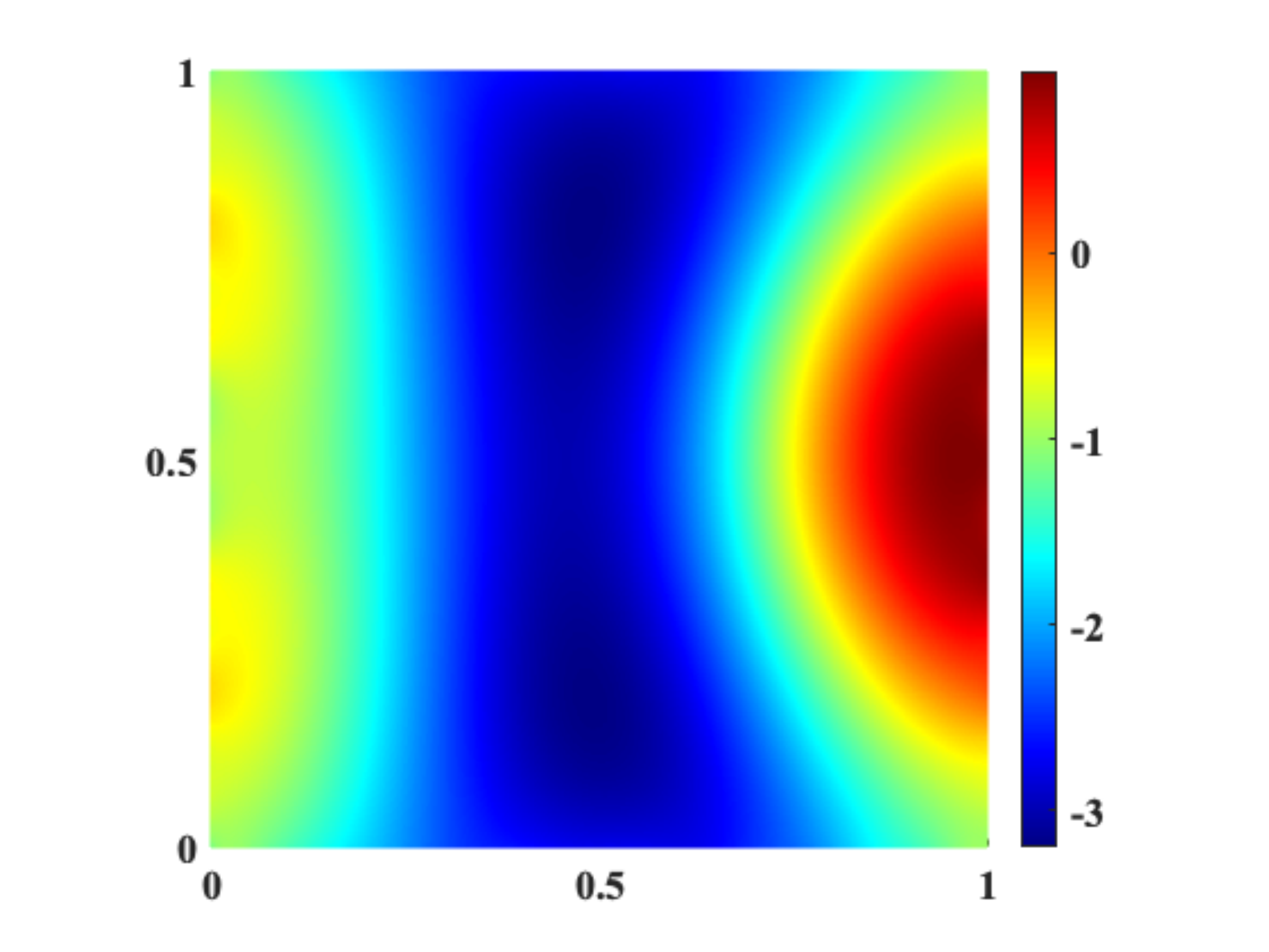}
\end{minipage}
\begin{minipage}[b]{0.45\textwidth}
	\includegraphics[width=\textwidth]{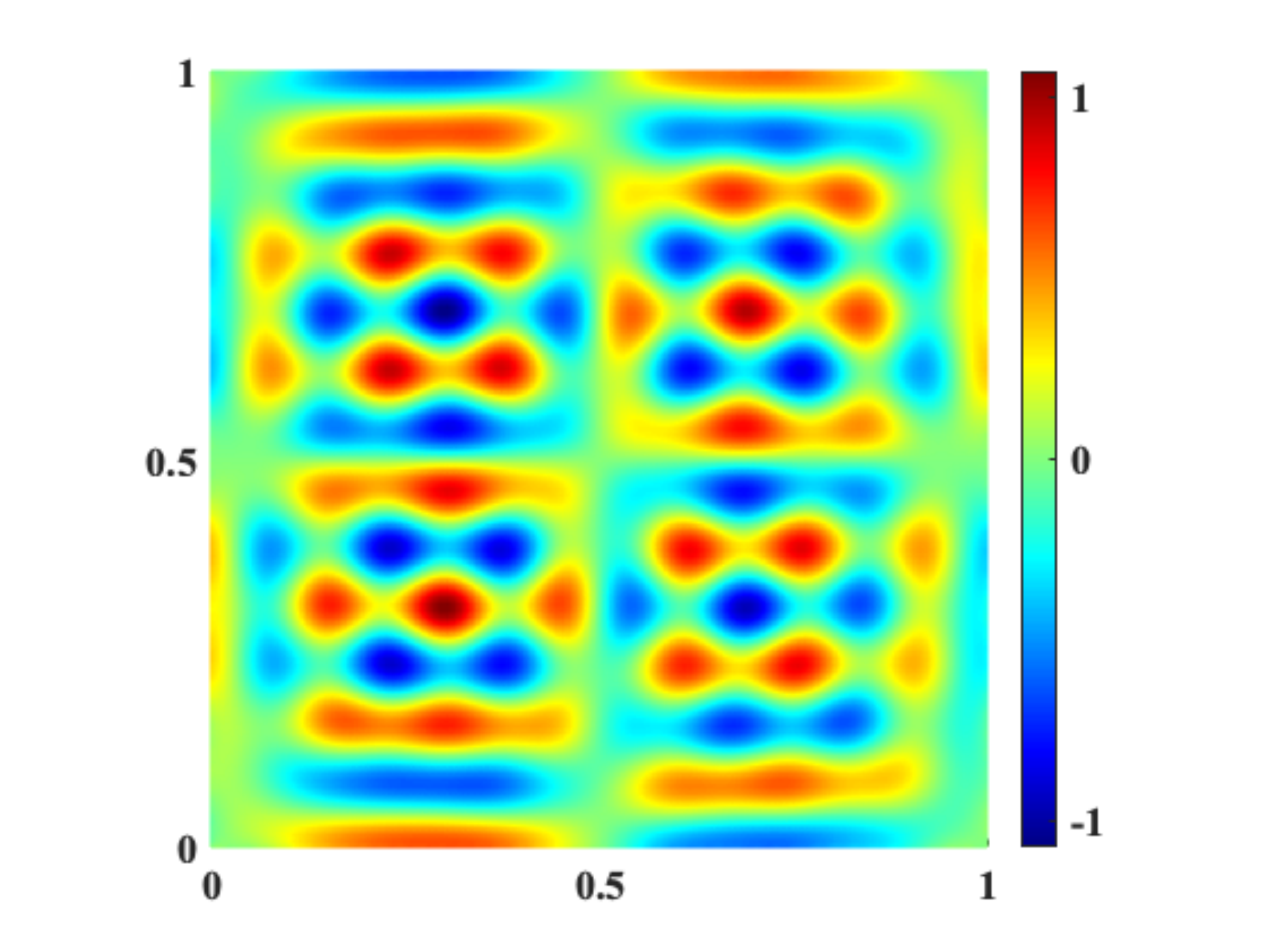}
\end{minipage}
\caption{Example 2. Behaviour of $\phi^P$ (left) and $\phi^S$ (right).}
\label{figrefphi}
\end{figure}
In Table \ref{cambio_h} we report the absolute errors of $\phi_{h_P}^P$ and $\phi_{h_S}^S$ obtained with $k_P=k_S=2$, by fixing $h_S = 1.1e-02$ and varying $h_P \in \{8.8e-02, 4.4e-02, 2.2e-02\}$. The error is calculated with respect to the approximate solution obtained with $k_P=k_S=2$ and a finer mesh sizes $\bar{h}_P=\bar{h}_S=1.1e-02$. As we can see, in terms of computational cost given by the number of dof reported in the last column, the most convenient choice is $h_P=4.4e-02$. Indeed, this mesh size allows us to obtain a satisfying accuracy for both potentials by saving about the $50\%$ of the dof. 

In Table \ref{cambio_k} we report the absolute errors of  $\phi_{h_P}^P$ and $\phi_{h_S}^S$ obtained with fixed mesh sizes $h_P=h_S=4.4e-02$, by fixing $k_S = 5$ and varying $k_P \in \{2, 3, 4\}$. The error is calculated with respect to the approximate solutions $\bar{\phi}^P_{h_P}$ and $\bar{\phi}^S_{h_S}$ obtained with $k_P=k_S=5$ and $h_P=h_S=4.4e-02$. Also in this case, to obtain a satisfying and comparable accuracy for both potentials, the most convenient approximation order is the intermediate $k_P =3$, with a dof saving of about the $30\%$. 
Further, from a comparison of the two Tables \ref{cambio_h} and \ref{cambio_k} in terms of accuracy on both solutions, we highlight that the advantageous strategy consists in decoupling the approximation orders and in using the higher one for the shear wave $\phi^S$.

\begin{table}[!h]
	\centering
			\begin{tabular}{cccc}
			\toprule
			$h_P$ & $\bigl\| \phi_{h_P}^P - \phi_{\bar{h}_P}^{P} \bigr\|_{\infty}$ & $\bigl\| \phi_{h_S}^S - \phi_{\bar{h}_S}^{S} \bigr\|_{\infty}$ & dof \vspace{0.1cm} \\ 	
			\toprule
			$8.8e-02$ & $8.3e-02$ & $4.1e-02$ & $67,138$\\
			$4.4e-02$ & $2.5e-02$ & $2.5e-02$ & $70,274$\\
			$2.2e-02$ & $2.3e-02$ & $3.2e-02$ & $82,690$\\
			$1.1e-02$ & $-$ & $-$ & $132,098$\\
			\bottomrule
		\end{tabular}
	\caption{Example 2. Absolute errors of $\phi_{h_P}^P$ and $\phi_{h_S}^S$ with $k_P=k_S=2$, by fixing $h_S = 1.1e-02$ and varying $h_P$.}
	\label{cambio_h}
\end{table}

\begin{table}[!h]
	\centering
			\begin{tabular}{cccc}
			\toprule
			$k_P$ & $\bigl\| \phi_{h_P}^P - \bar{\phi}_{h_P}^{P} \bigr\|_{\infty}$ & $\bigl\| \phi_{h_S}^S - \bar{\phi}_{h_S}^{S} \bigr\|_{\infty}$ & dof \vspace{0.1cm} \\ 	
			\toprule
			$2$ & $3.2e-02$ & $7.3e-03$ & $24,002$\\
			$3$ & $7.2e-03$ & $5.8e-03$ & $28,162$\\
			$4$ & $9.6e-03$ & $4.9e-03$ & $33,346$\\
			$5$ & $-$ & $-$ & $39,554$\\
			\bottomrule
		\end{tabular}
	\caption{Example 2. Absolute errors of $\phi_{h_P}^P$ and $\phi_{h_S}^S$ with $h_P=h_S=4.4e-02$, by fixing $k_S = 5$ and varying $k_P$.}
	\label{cambio_k}
\end{table}

\

\noindent \textbf{Example 3.} In this example we consider Problem \eqref{pb_iniziale} with a non-trivial vector source $\bbf$, for which the corresponding Helmholtz-Hodge decomposition $\bbf = \nabla f^P + \bccurl f^S$ is not analytically given. Hence we proceed numerically, by determining  $f_{h_P}^P \in Q_{h_P}^{k_P}$ and $f_{h_S}^S \in Q_{h_S}^{k_S}$, VEM approximations of $f^P$ and $f^S$ obtained by solving Problems \eqref{pbdualf} and \eqref{pbduald}, respectively. These approximations will be then used in the right hand side term of \eqref{pb_strong}. In particular, the non-homogeneous Neumann problem in the unknown $f^P$ is reformulated in terms of a standard variational formulation (see, e.g. \cite[Theorem 4.1]{Steinbach2008}) and then discretized by the VEM, as follows: find $f_{h_P}^P \in Q_{h_P}^{k_P}$ such that 
\begin{align}\label{numschemefP}
	a_P\left(f_{h_P}^P , v_{h_P}^P\right) +  \langle 1,f_{h_P}^P \rangle_\Gamma \langle 1,v_{h_P}^P \rangle_\Gamma =  -\left (\ddiv \bbf,\Pi_{k_P^*}^0 v_{h_P}^P\right)_{L^2(\Omega)}  + \langle \bbf \cdot \bn,v_{h_P}^P \rangle_\Gamma
\end{align}
for all $v_{h_P}^P \in Q_{h_P}^{k_P}$. Besides, the homogeneous Dirichlet problem in the unknown $f^S$ is reformulated and approximated as (see \cite{BeiraoBrezziCangianiManziniMariniRusso2013}): find $f_{h_S}^{S} \in Q_{h_S}^{k_S} \cap H_0^1(\Omega)$ such that 
\begin{equation} \label{numschemefS}
	a_S\left(f_{h_S}^S , v_{h_S}^S\right) =\left (\ccurl \bbf,\Pi_{k_S^*}^0 v_{h_S}^S\right)_{L^2(\Omega)}
\end{equation}
for all $v_{h_S}^S \in Q_{h_S}^{k_S} \cap H_0^1(\Omega)$.
We remark that, proceeding as in \cite{BeiraoBrezziCangianiManziniMariniRusso2013} and \cite{BrennerGuanSung2017}, it is possible to prove that if $f^\diamond$, $\diamond = P,S$,
is smooth enough, then
\begin{equation*}
	\| f^\diamond - f_{h_\diamond}^\diamond \|_{L^2(\Omega)} + h_\diamond| f^\diamond - f_{h_\diamond}^\diamond |_{H^1(\Omega)} \apprle h_\diamond^{k_\diamond+1} \| \bbf \|_{\bH^{k_\diamond+1}(\Omega)}.
\end{equation*}
Therefore, once $f_{h_\diamond}^\diamond$ 
has been retrieved, we compute the right hand side \eqref{eq:Lop_disc} as (recall $k_\diamond^* = \max\{1,k_\diamond-2\}$)
\begin{equation*}
	( f_{h_\diamond}^\diamond,  \Pi^0_{k_\diamond^*} v_{h_\diamond})_{L^2(\Omega)} = (\Pi_{k_\diamond^*}^0 f_{h_\diamond}^\diamond,  \Pi_{k_\diamond^*}^0 v_{h_\diamond})_{L^2(\Omega)}.
\end{equation*}
We will show that the numerical procedure adopted for the approximation of the right hand sides does not affect the convergence order of the global scheme.
To this aim we consider the L-shaped domain $\Omega = \Omega_1 \setminus \Omega_2$ where $\Omega_1 = (0,2) \times (0,2)$ and $\Omega_2 = [0.5,2) \times [0.5,2)$, and the parameters $\mu = 5$, $\lambda = 1$, $\rho = 10$ and $\kappa = 1$. The vector source $\bbf$, whose components are represented in Figure \ref{fig:f12}, and the Dirichlet boundary condition $\bg$ are taken  accordingly to the solution
\begin{equation} \label{ff_ex3}
	\bu(x_1,x_2) = \begin{bmatrix} e^{-100((x_1-0.25)^2 + (x_2-1.75)^2)} \\ e^{-100((x_1-1.75)^2 + (x_2-0.25)^2)} \end{bmatrix}.
\end{equation}
In Figure \ref{fig:f12PS} we show the behaviour of $f_{h_P}^P$ and $f_{h_S}^S$, approximations of $f^P$ and $f^S$ obtained by solving \eqref{numschemefP} and \eqref{numschemefS}, respectively.
\begin{figure}[h]
\centering
\begin{minipage}[b]{0.45\textwidth}
	\includegraphics[width=\textwidth]{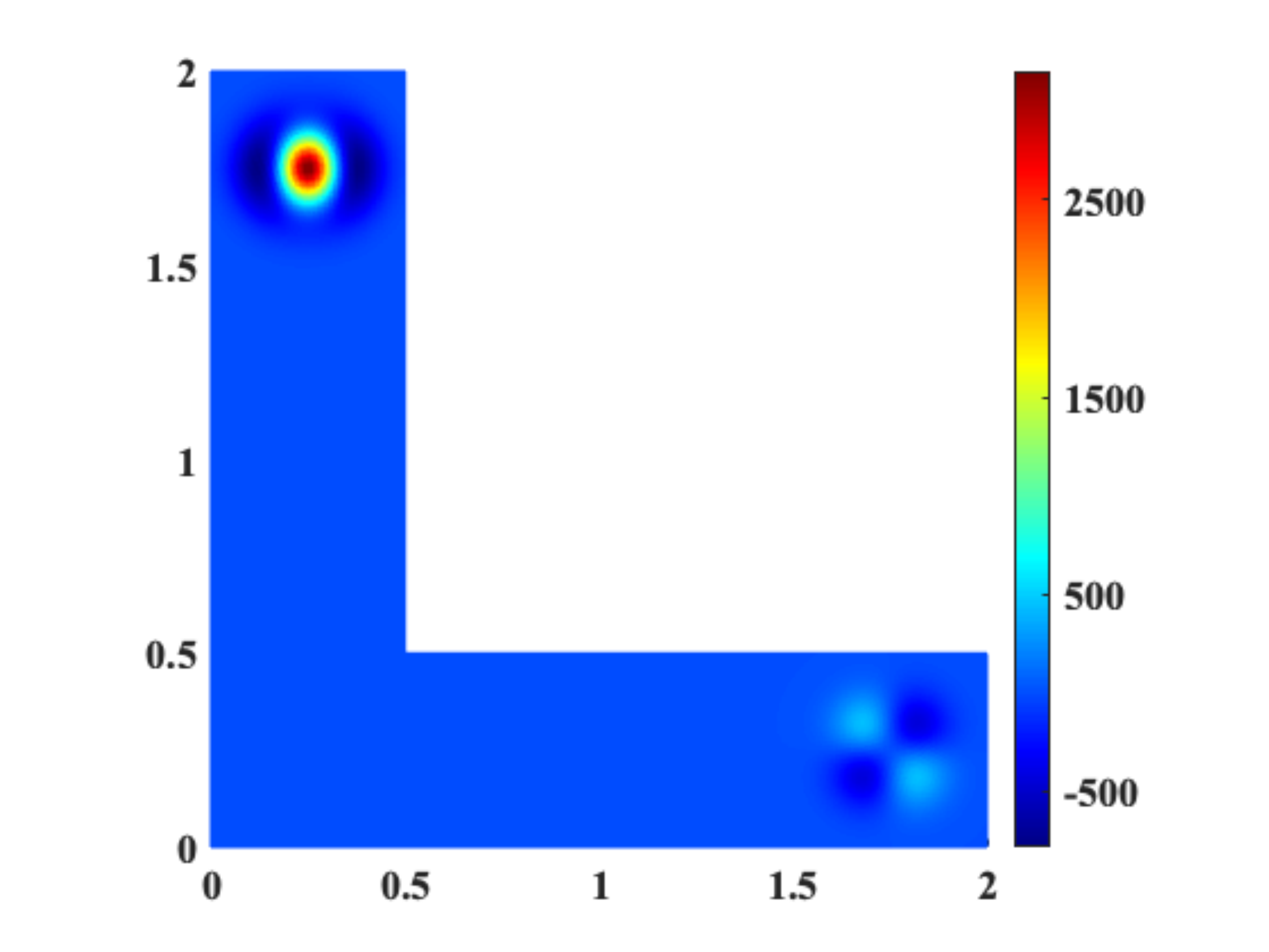}
\end{minipage}
\begin{minipage}[b]{0.45\textwidth}
	\includegraphics[width=\textwidth]{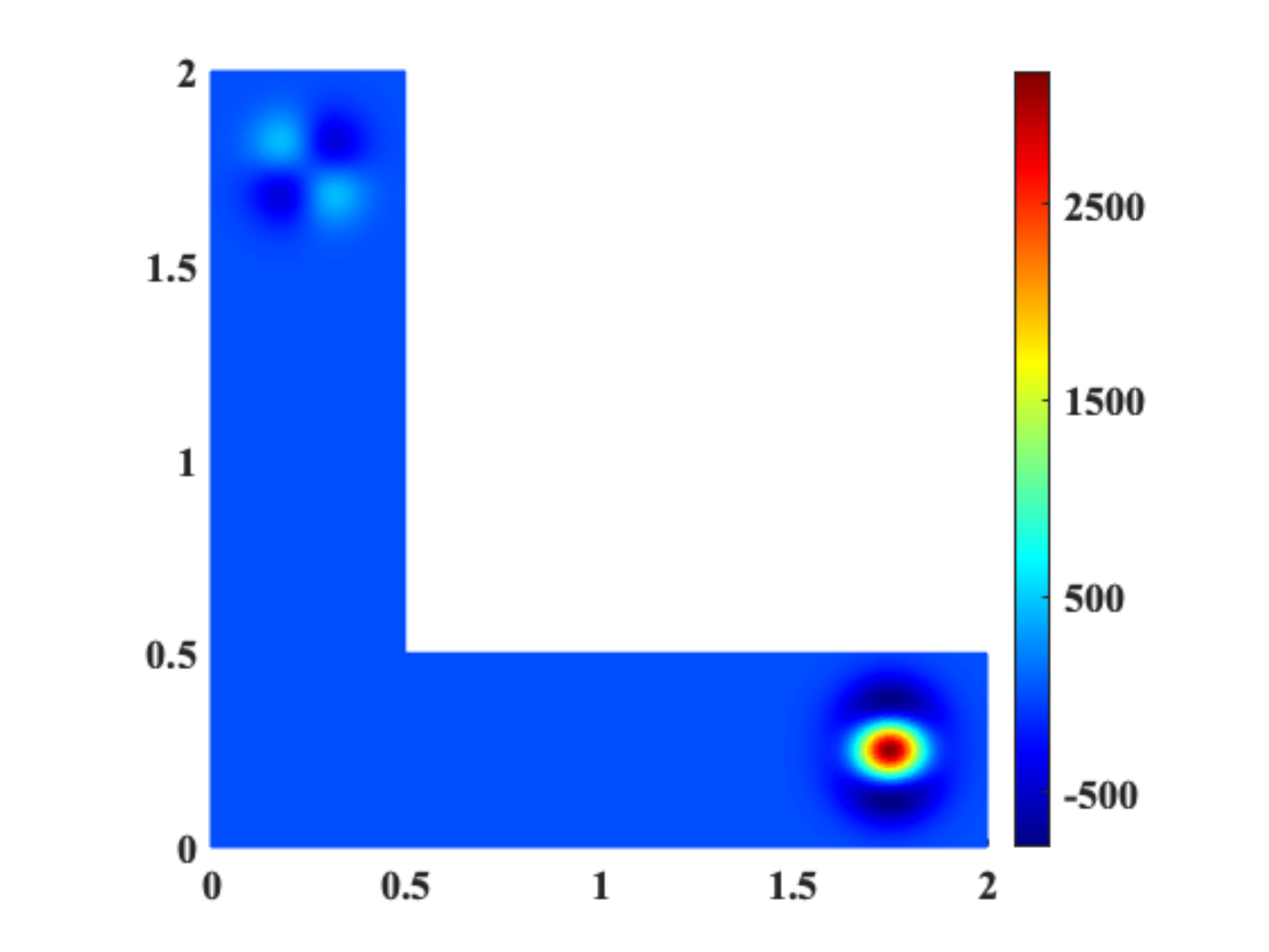}
\end{minipage}
\caption{Example 3. Behaviour of the source components $f_1$ (left) and $f_2$ (right).} 
\label{fig:f12}
\end{figure}
\begin{figure}[h!]
\centering
\begin{minipage}[b]{0.45\textwidth}
	\includegraphics[width=\textwidth]{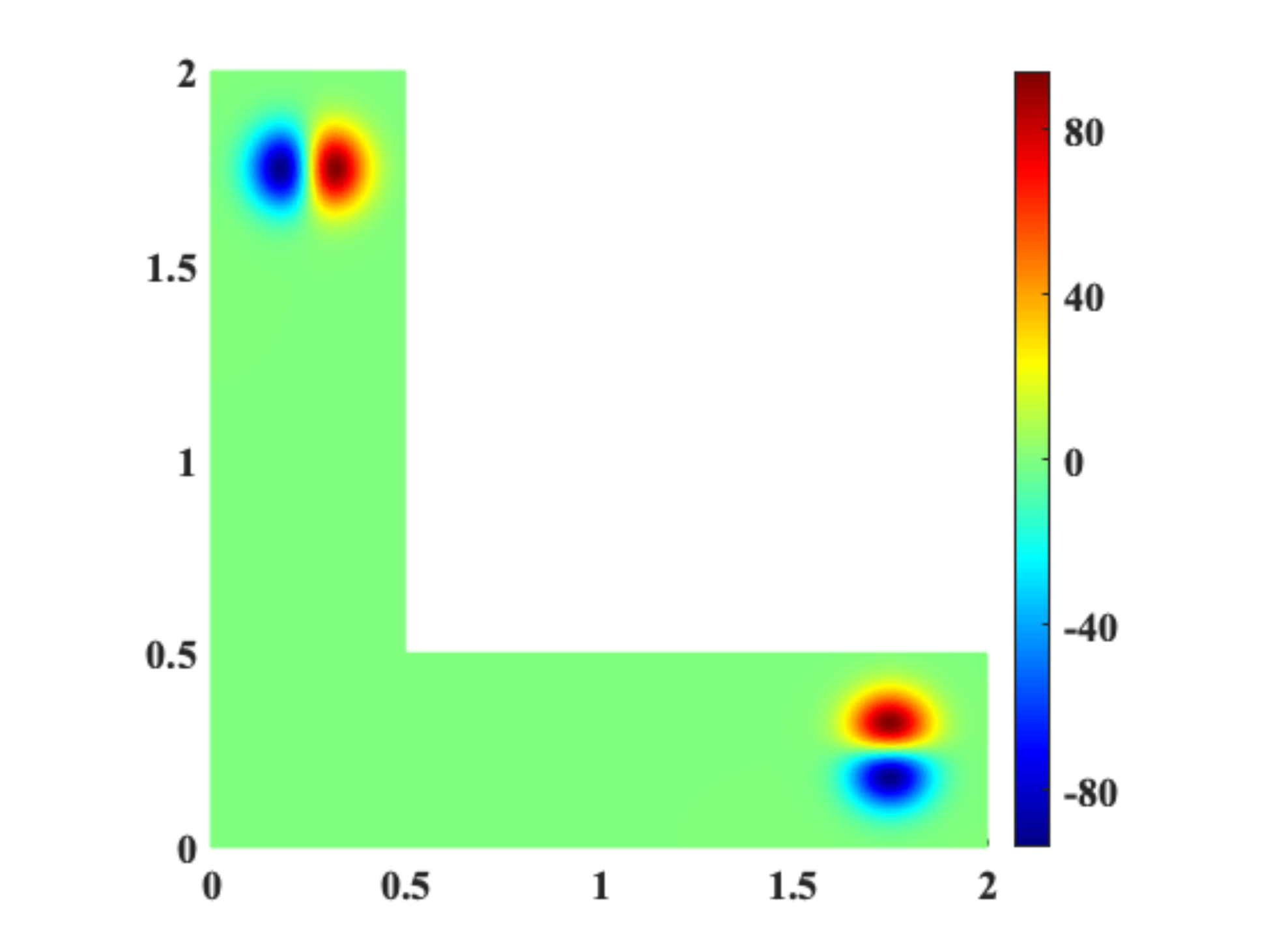}
\end{minipage}
\begin{minipage}[b]{0.45\textwidth}
	\includegraphics[width=\textwidth]{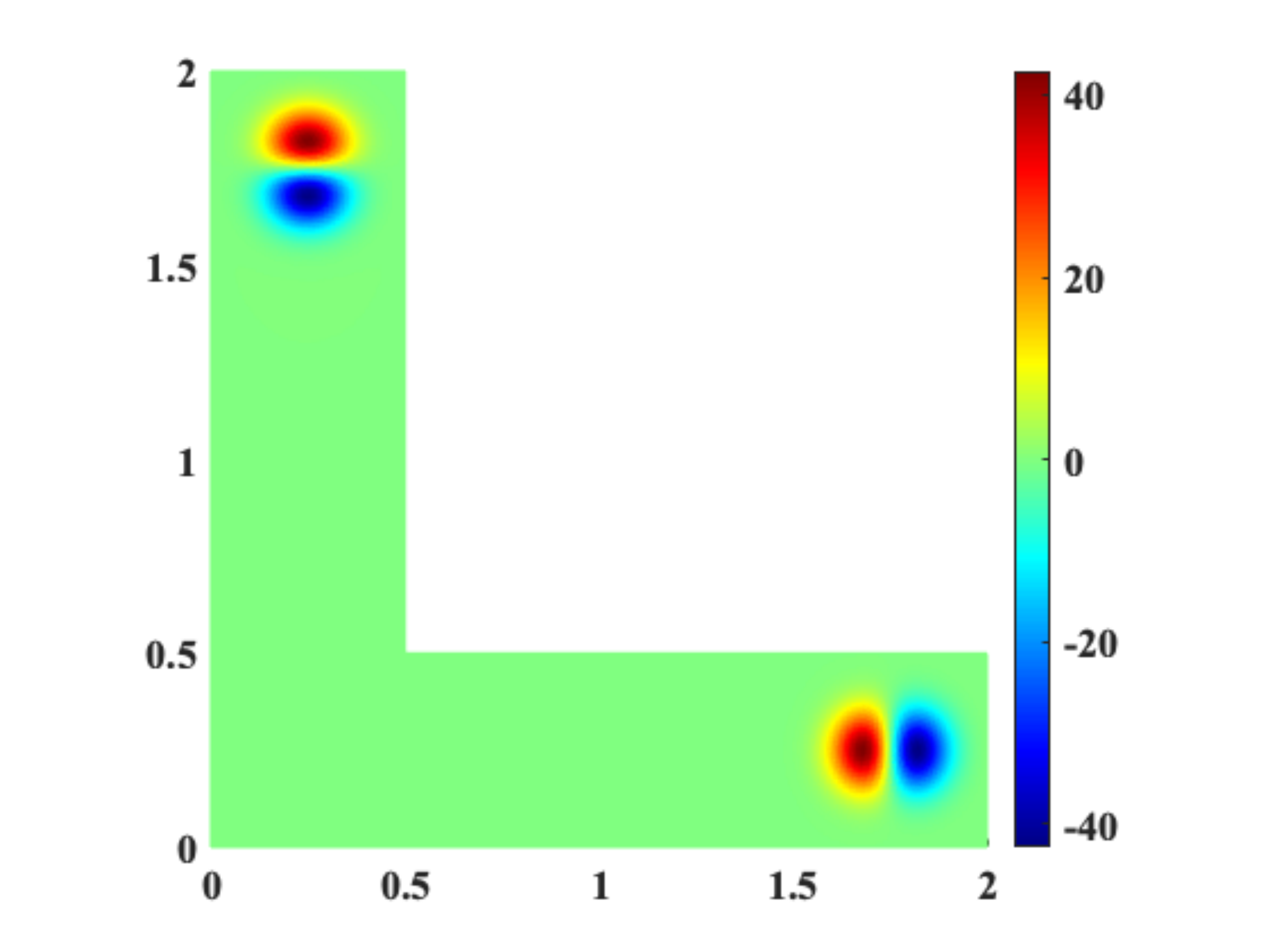}
\end{minipage}
\caption{Example 3. Behaviour of $f_{h_P}^P$ (left) and $f_{h_S}^S$ (right).} \label{fig:f12PS}
\end{figure}
We compare the numerical solutions computed by applying the linear \textit{vector} VEM and the second order ($k_P=k_S=2$) \textit{scalar} one associated to the same tessellation with mesh size $h=h_P=h_S=3.19e-02$. In Figures \ref{fig:u1vettscal} and \ref{fig:u2vettscal} we plot the absolute errors of the two entries of the approximate solutions $\bu_h$ and $\bu_{\bh}^C$, with respect to the exact one. As expected, the maximum absolute error for the linear vector procedure is of the same order of magnitude of the quadratic scalar one. Furthermore, we point out that the VEM matrices related to the solution of \eqref{pbdualf} and \eqref{pbduald} are reused in the subsequent resolution of the scalar scheme, so that the extra computational cost to retrieve the approximations of $f^P$ and $f^S$ is negligible with respect to the overall one.
\begin{figure}[h]
 \centering
 \begin{minipage}[b]{0.38\textwidth}
 	\includegraphics[width=\textwidth]{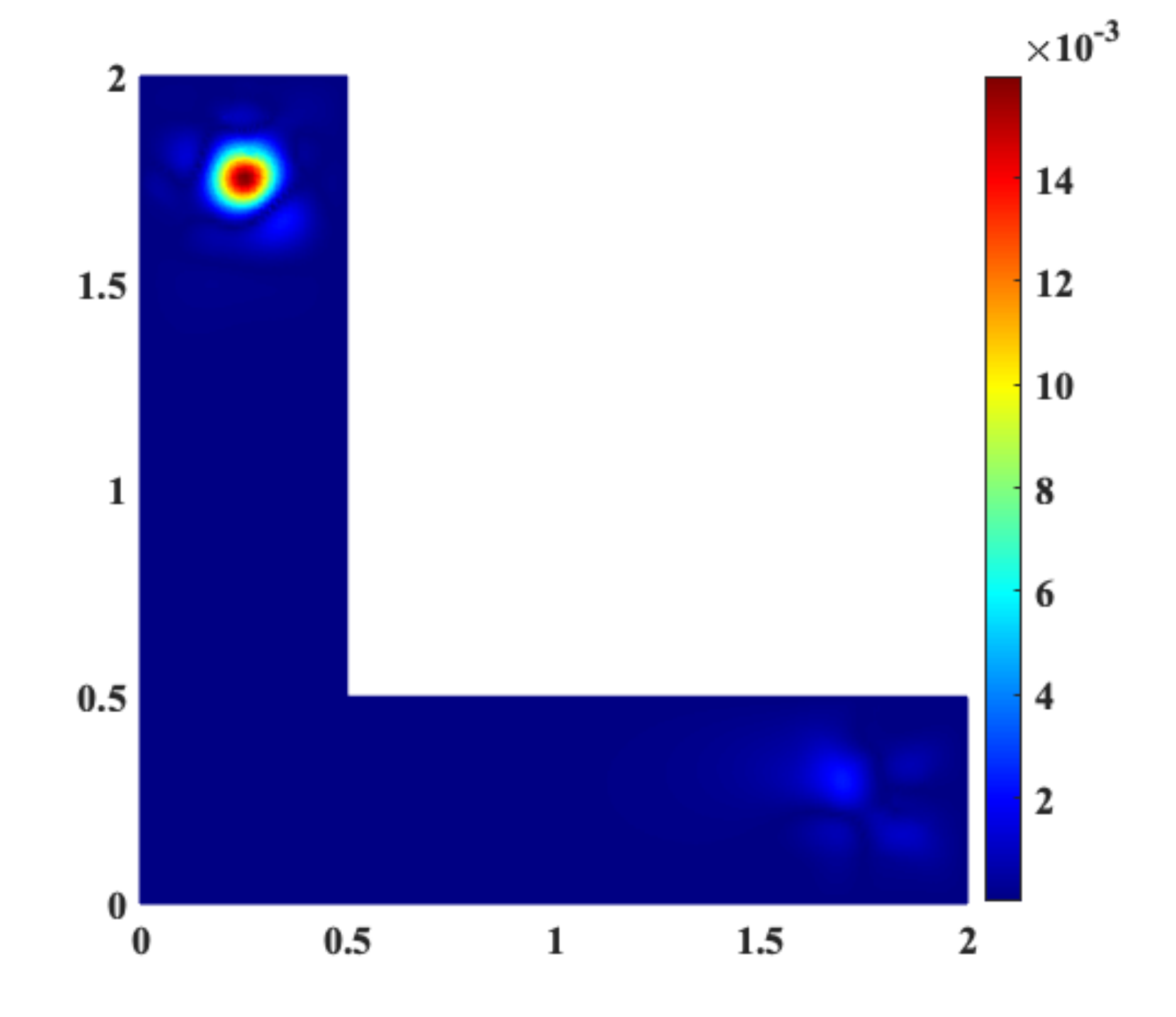}
\end{minipage}
\begin{minipage}[b]{0.44\textwidth}
	\includegraphics[width=\textwidth]{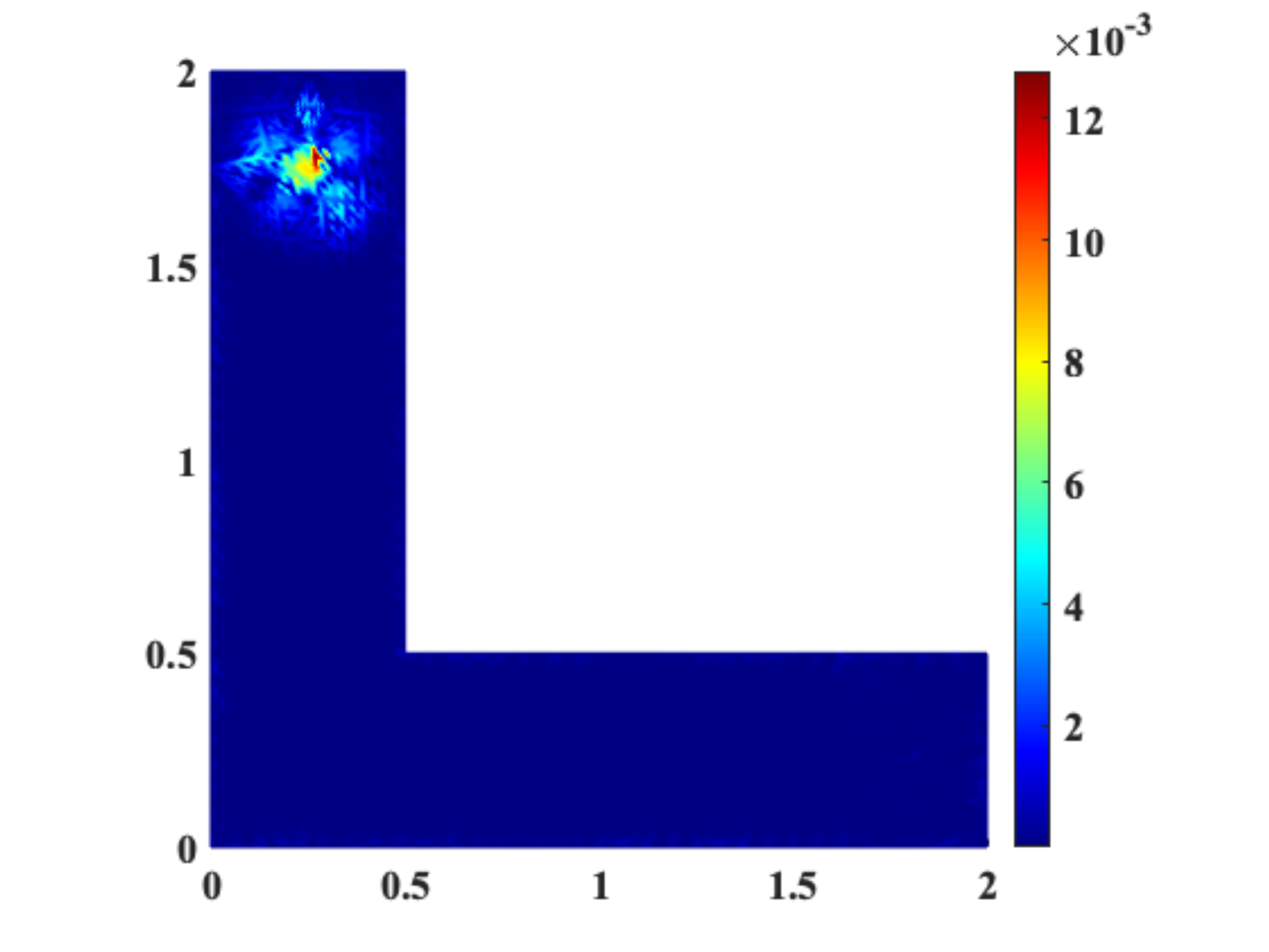}
\end{minipage}
\caption{Example 3. Absolute errors of $u_1$ obtained with the \textit{vector} VEM (left) and the \textit{scalar} VEM (right).} \label{fig:u1vettscal}
\end{figure} 
\begin{figure}[h!]
\centering
\begin{minipage}[b]{0.45\textwidth}
	\includegraphics[width=\textwidth]{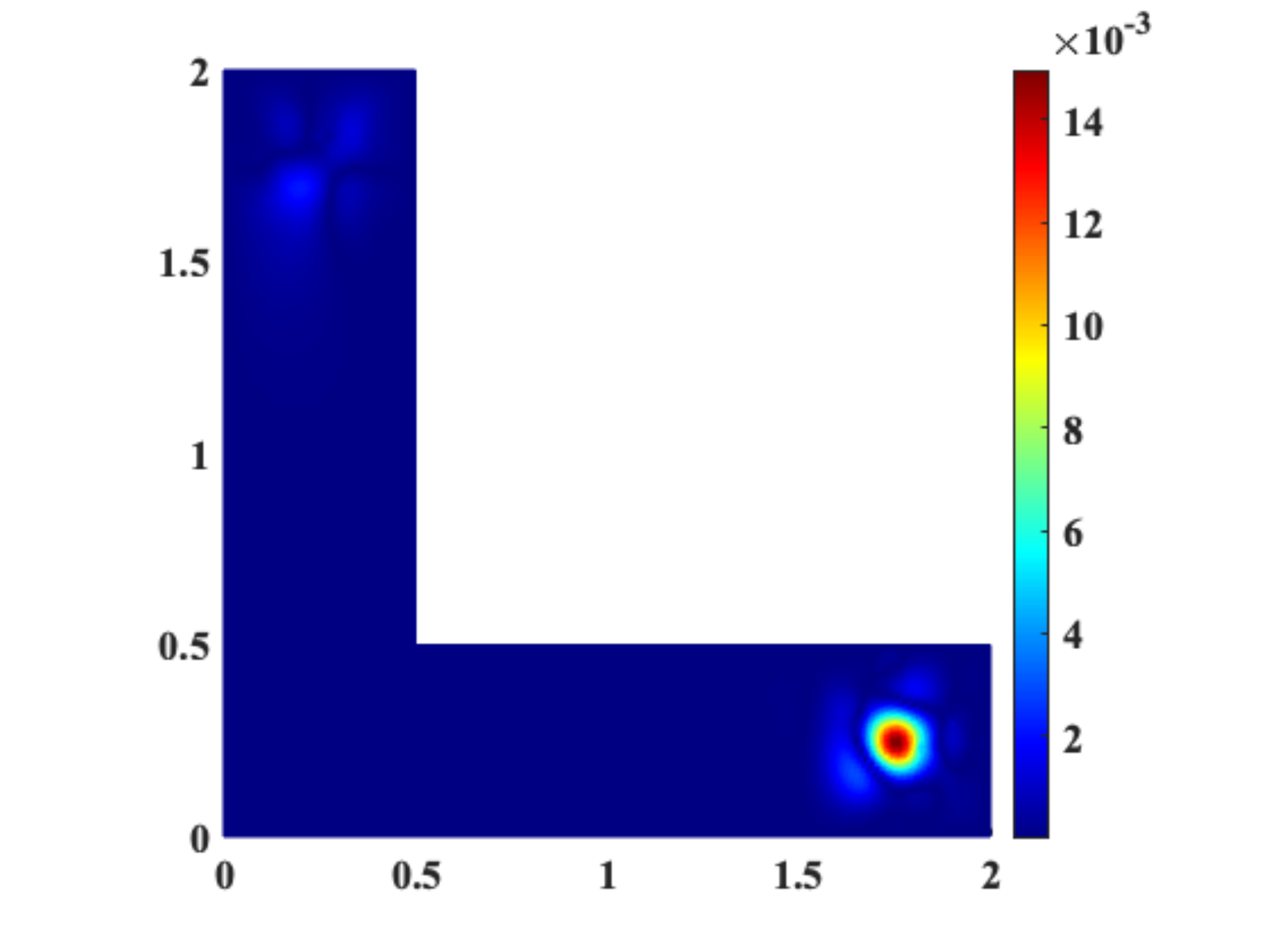}
\end{minipage}
\begin{minipage}[b]{0.445\textwidth}
	\includegraphics[width=\textwidth]{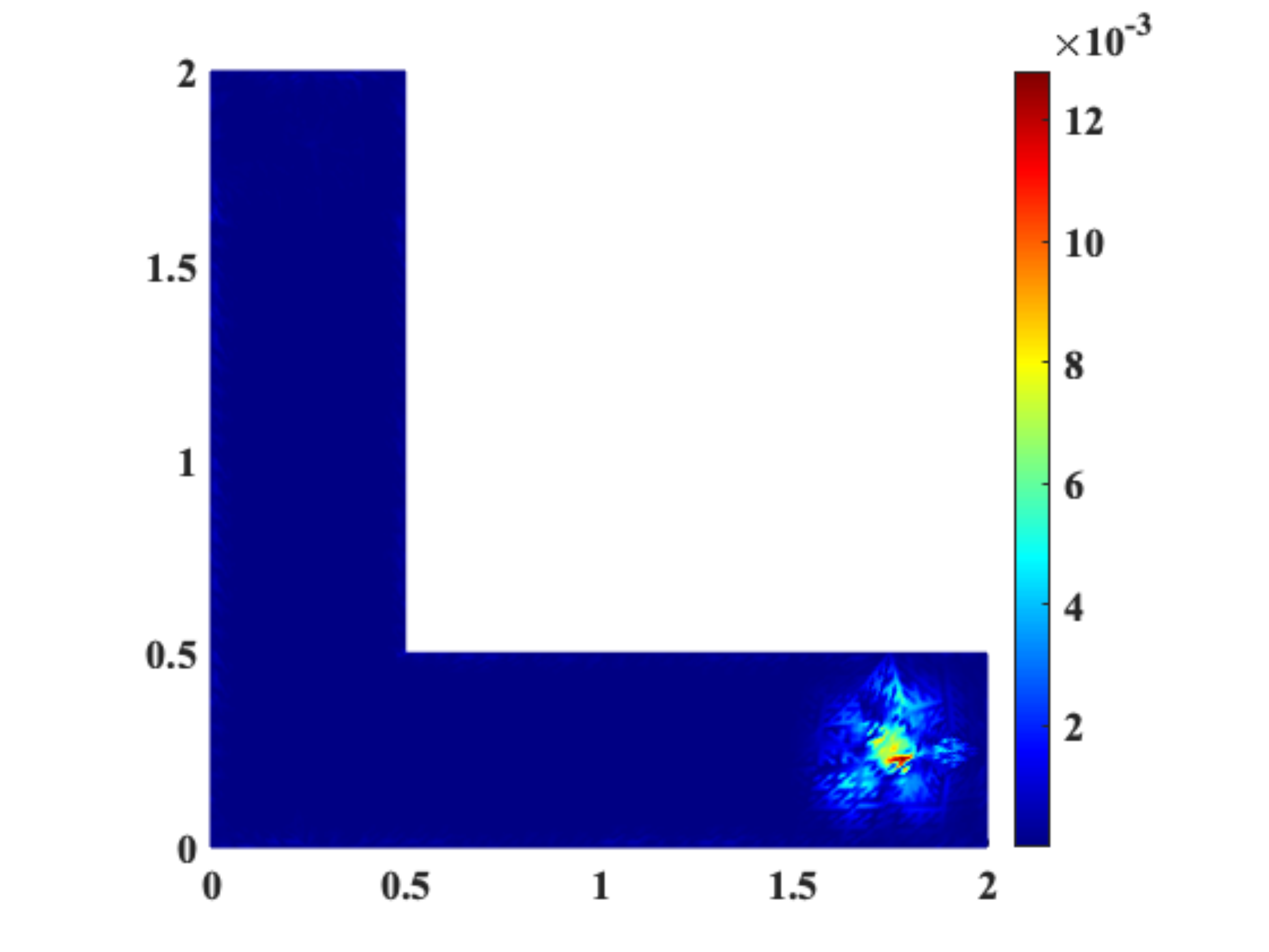}
\end{minipage}
\caption{Example 3. Absolute errors of $u_2$ obtained with the \textit{vector} VEM (left) and the \textit{scalar} VEM (right).}  \label{fig:u2vettscal}
\end{figure}

\

\noindent\textbf{Example 4.} In this last example we aim at showing the feasibility of the \textit{scalar} approach to deal with curved geometries, by using well-established tools of the curved VEM for scalar Helmholtz problems.
Indeed, for this test, we have applied the curvilinear version of the VEM used in  \cite{DesiderioFallettaFerrariScuderi20221} and, for the construction of our final linear system, we have used the same VEM matrices therein involved. This aspect turns out to be an advantage of the \textit{scalar} approach with respect to the \textit{vector} one for which, to the best of our knowledge, nowadays in literature there are no results related to curved geometries, either theoretical or numerical.
   
To test the convergence rate, we consider Problem \eqref{pb_iniziale} defined in the unit disk $\Omega$, with $\kappa=1$, $\rho=1$, $\mu=1$ and $\lambda=10$.
We choose $\mathbf{f} = \nabla f^P + \bccurl f^S$ with
\begin{align*}
	f^P(x_1,x_2) & = -24e^{x_2}(\cos x_2 -x_1^2\sin x_2 )-x_1^2e^{x_2}\cos x_2 ,
	\\ f^S(x_1,x_2) & = -2e^{x_1}(\sin x_1 +x_2^2\cos x_1)-x_2^2e^{x_1}\sin x_1,
\end{align*}
so that the corresponding solution is
\begin{align*}
	\bu(x_1,x_2) =
	\begin{bmatrix}
		2(x_1e^{x_2}\cos x_2 +x_2e^{x_1}\sin x_1 )
		\\ x_1^2e^{x_2}(\cos x_2 -\sin x_2) - x_2^2e^{x_1}(\cos x_1 +\sin x_1)
	\end{bmatrix}
\end{align*}
and the $P-$ and $S-$ waves are
\begin{equation*}
	\phi^P(x_1,x_2) = x_1^2e^{x_2}\cos x_2, \quad \phi^S(x_1,x_2) = x_2^2e^{x_1}\sin x_1 .
\end{equation*}
In Figure \ref{fig:L2circ} we plot the absolute $L^2$-norm error of the solutions obtained with approximation orders $k_P = k_S \in \{1,2,3,4,5\}$, with respect to the mesh size $h = h_P = h_S$. Since the curvilinear VEM permits to avoid the approximation of the geometry, the error on the solution depends only on the VEM approximation and, hence, the convergence rate is the expected optimal one.

\begin{figure}[h]
	\centering
	\begin{minipage}[b]{0.5\textwidth}
		\includegraphics[width=\textwidth]{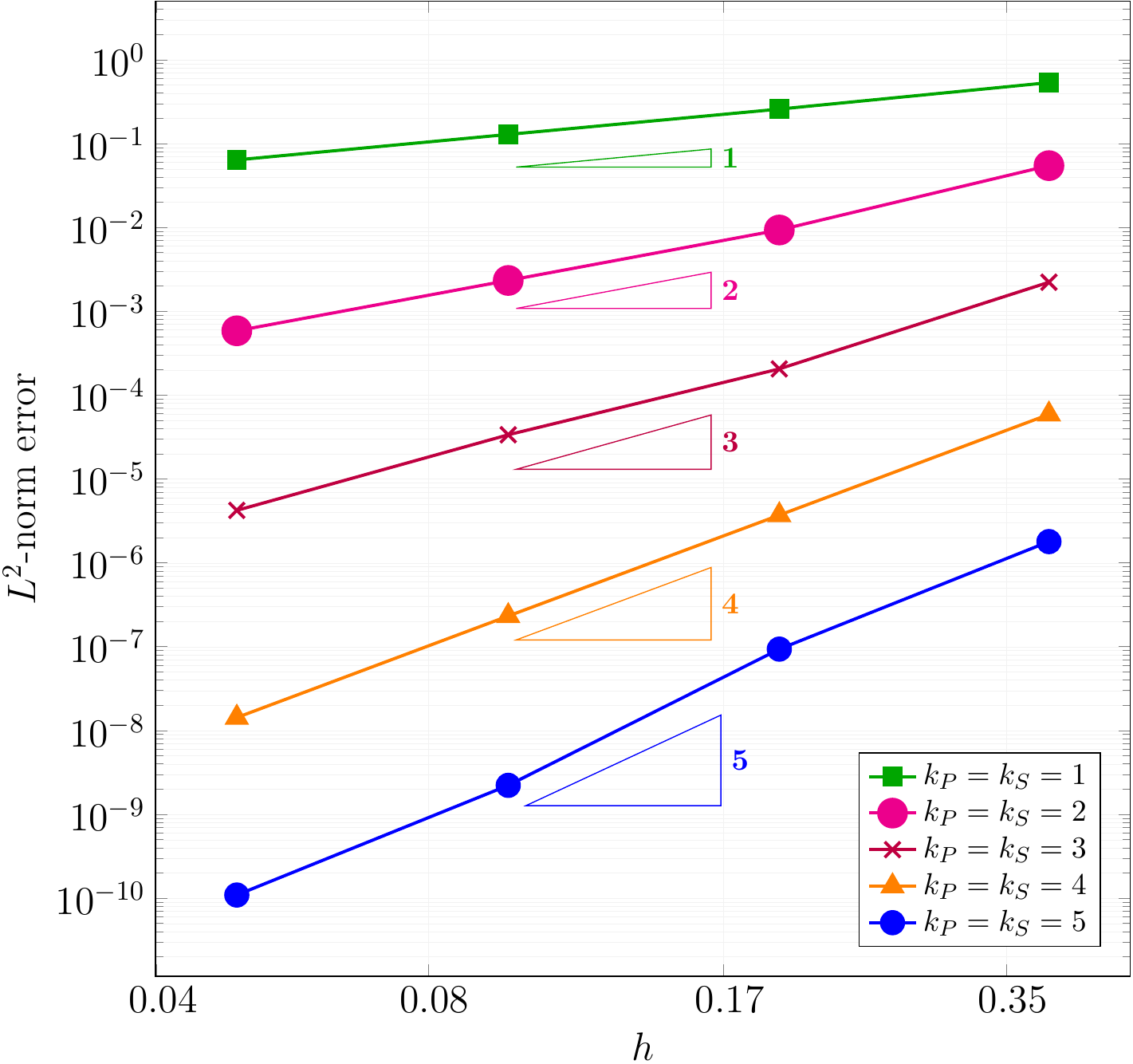}
	\end{minipage}
	\caption{Example 4. $L^2$-norm absolute error for the \textit{scalar} curved VEM by varying $k_P=k_S$.}  \label{fig:L2circ}
\end{figure}

\section{Conclusion}
In this work, we have proposed a novel approach for the numerical simulation of two dimensional time-harmonic elastodynamics problems. It consists in reformulating the original PDE in terms of two coupled wave equations involving, as new unknowns, the $P-$ and $S-$ waves scalar potentials.

We have provided the stability analysis of the scalar system by means of the (non classical) T-coercivity of the bilinear form associated with the variational formulation and, for its approximation, we have applied a virtual element method. Convergence estimates have been derived and confirmed by numerical test.

This approach turns out to be a valid and competitive alternative to the vector displacement-based one. In particular, an advantage that we have highlighted for the scalar formulation is its feasibility in using different approximation orders and/or mesh sizes of the domain tessellation. This aspect revealed to be crucial when dealing with materials in which $P-$ and $S-$ waves are associated to different wave numbers, since it permits to tune the approximation parameters accordingly. 
Furthermore, from the implementation view point, the proposed approach allowed us to use the well-established in-house software developed for standard Helmholtz problems, easily including curved geometries in the numerical investigation.

As a future development, we aim at extending the analysis of this approach to curved VEM as well as to exterior elastic problems, combining the interior VEM with a boundary one, both in the time-harmonic and in the space-time case. In particular, for the space-time case, we aim at making use of the numerical scheme for the classical two dimensional wave equation proposed in \cite{DesiderioFallettaFerrariScuderi2023}.

\section{Acknowledgments}
This work was performed as part of the GNCS-INDAM 2022 research program ``Metodi numerici avanzati VEM e VEM-BEM per PDEs: propriet{\`a} teoriche e aspetti computazionali''. The second author was partially supported by MIUR grant ``Dipartimenti di Eccellenza 2018-2022", CUP E11G18000350001.


\begin{thebibliography}{50}
\expandafter\ifx\csname url\endcsname\relax
  \def\url#1{\texttt{#1}}\fi
\expandafter\ifx\csname urlprefix\endcsname\relax\def\urlprefix{URL }\fi
\expandafter\ifx\csname href\endcsname\relax
  \def\href#1#2{#2} \def\path#1{#1}\fi

\bibitem{BurelImperialeJoly2012}
A.~Burel, S.~Imperiale, P.~Joly, Solving the {H}omogeneous {I}sotropic {L}inear
  {E}lastodynamics {E}quations {U}sing {P}otentials and {F}inite {E}lements.
  {T}he {C}ase of the {R}igid {B}oundary {C}ondition, Numer. Analys. Appl.
  5~(2) (2012) 136--143.

\bibitem{AlbellaImperialJolyRodriguez2018}
J.~Albella~M., S.~Imperale, P.~Joly, J.~Rodr\'{\i}guez, Solving 2{D} linear
  isotropic elastodynamics by means of scalar potentials: a new challenge for
  finite elements, J. Sci. Comput. 77~(3) (2018) 1832--1873.

\bibitem{AlbellaMartinezImperialeJolyRodriguez2021}
J.~Albella~Mart\'{\i}nez, S.~Imperiale, P.~Joly, J.~Rodr\'{\i}guez, Numerical
  analysis of a method for solving 2{D} linear isotropic elastodynamics with
  traction free boundary condition using potentials and finite elements, Math.
  Comp. 90~(330) (2021) 1589--1636.

\bibitem{FallettaMonegatoScuderi2019}
S.~Falletta, G.~Monegato, L.~Scuderi, Two boundary integral equation methods
  for linear elastodynamics problems on unbounded domains, Comput. Math. Appl.
  78~(12) (2019) 3841--3861.

\bibitem{FallettaMonegatoScuderi2022}
S.~Falletta, G.~Monegato, L.~Scuderi, Two {FEM}-{BEM} methods for the numerical
  solution of 2{D} transient elastodynamics problems in unbounded domains,
  Comput. Math. Appl. 114 (2022) 132--150.

\bibitem{BonnetBenDhiaCiarletZwolf2010}
A.~S. Bonnet-Ben~Dhia, P.~Ciarlet, Jr., C.~M. Zw\"{o}lf, Time harmonic wave
  diffraction problems in materials with sign-shifting coefficients, J. Comput.
  Appl. Math. 234~(6) (2010) 1912--1919.

\bibitem{Ciarlet2012}
P.~Ciarlet, Jr., {$T$}-coercivity: application to the discretization of
  {H}elmholtz-like problems, Comput. Math. Appl. 64~(1) (2012) 22--34.

\bibitem{BeiraoBrezziMarini2013}
L.~Beir\~{a}o~da Veiga, F.~Brezzi, L.~D. Marini, Virtual elements for linear
  elasticity problems, SIAM J. Numer. Anal. 51~(2) (2013) 794--812.

\bibitem{GainTalischiPaulino2014}
A.~L. Gain, C.~Talischi, G.~H. Paulino, On the virtual element method for
  three-dimensional linear elasticity problems on arbitrary polyhedral meshes,
  Comput. Methods Appl. Mech. Engrg. 282 (2014) 132--160.

\bibitem{BeiraoLovadinaMora2015}
L.~Beir\~{a}o~da Veiga, C.~Lovadina, D.~Mora, A virtual element method for
  elastic and inelastic problems on polytope meshes, Comput. Methods Appl.
  Mech. Engrg. 295 (2015) 327--346.

\bibitem{ArtioliBeiraoLovadinaSacco2017}
E.~Artioli, L.~Beir\~{a}o~da Veiga, C.~Lovadina, E.~Sacco, Arbitrary order 2{D}
  virtual elements for polygonal meshes: part {I}, elastic problem, Comput.
  Mech. 60~(3) (2017) 355--377.

\bibitem{ArtioliBeiraoLovadinaSacco20172}
E.~Artioli, L.~Beir\~{a}o~da Veiga, C.~Lovadina, E.~Sacco, Arbitrary order 2{D}
  virtual elements for polygonal meshes: part {II}, inelastic problem, Comput.
  Mech. 60~(4) (2017) 643--657.

\bibitem{AntoniettiManziniMazzieriMouradVerani2021}
P.~F. Antonietti, G.~Manzini, I.~Mazzieri, H.~M. Mourad, M.~Verani, The
  arbitrary-order virtual element method for linear elastodynamics models:
  convergence, stability and dispersion-dissipation analysis, Internat. J.
  Numer. Methods Engrg. 122~(4) (2021) 934--971.

\bibitem{BeiraoBrezziMariniRusso2014}
L.~Beir\~{a}o~da Veiga, F.~Brezzi, L.~D. Marini, A.~Russo, The hitchhiker's
  guide to the virtual element method, Math. Models Methods Appl. Sci. 24~(8)
  (2014) 1541--1573.

\bibitem{DesiderioFallettaScuderi2021}
L.~Desiderio, S.~Falletta, L.~Scuderi, A {V}irtual {E}lement {M}ethod coupled
  with a {B}oundary {I}ntegral {N}on {R}eflecting condition for 2{D} exterior
  {H}elmholtz problems, Comput. Math. Appl. 84 (2021) 296--313.

\bibitem{GiraultRaviart1986}
V.~Girault, P.-A. Raviart, Finite element methods for {N}avier-{S}tokes
  equations, Vol.~5 of Springer Series in Computational Mathematics,
  Springer-Verlag, Berlin, 1986, theory and algorithms.

\bibitem{Burel2014}
A.~Burel, {Contributions {\`a} la simulation num{\'e}rique en
  {\'e}lastodynamique : d{\'e}couplage des ondes P et S, mod{\`e}les
  asymptotiques pour la travers{\'e}e de couches minces}, Ph.D. thesis.

\bibitem{AmroucheBernardiDaugeGirault1998}
C.~Amrouche, C.~Bernardi, M.~Dauge, V.~Girault, Vector potentials in
  three-dimensional non-smooth domains, Math. Methods Appl. Sci. 21~(9) (1998)
  823--864.

\bibitem{Buffa2005}
A.~Buffa, Remarks on the discretization of some noncoercive operator with
  applications to heterogeneous {M}axwell equations, SIAM J. Numer. Anal.
  43~(1) (2005) 1--18.

\bibitem{SayasBrownHassel2019}
F.-J. Sayas, T.~S. Brown, M.~E. Hassel, Variational techniques for elliptic
  partial differential equations, CRC Press, Boca Raton, FL, 2019, theoretical
  tools and advanced applications.

\bibitem{Grisvard1985}
P.~Grisvard, Elliptic problems in nonsmooth domains, Vol.~24 of Monographs and
  Studies in Mathematics, Pitman (Advanced Publishing Program), Boston, MA,
  1985.

\bibitem{DiNezzaPalatucciValdinoci2012}
E.~Di~Nezza, G.~Palatucci, E.~Valdinoci, Hitchhiker's guide to the fractional
  {S}obolev spaces, Bull. Sci. Math. 136~(5) (2012) 521--573.

\bibitem{AhmadAlsaediBrezziMariniRusso2013}
B.~Ahmad, A.~Alsaedi, F.~Brezzi, L.~D. Marini, A.~Russo, Equivalent projectors
  for virtual element methods, Comput. Math. Appl. 66~(3) (2013) 376--391.

\bibitem{BeiraoBrezziCangianiManziniMariniRusso2013}
L.~Beir\~{a}o~da Veiga, F.~Brezzi, A.~Cangiani, G.~Manzini, L.~D. Marini,
  A.~Russo, Basic principles of virtual element methods, Math. Models Methods
  Appl. Sci. 23~(1) (2013) 199--214.

\bibitem{BeiraoLovadinaRusso2017}
L.~Beir\~{a}o~da Veiga, C.~Lovadina, A.~Russo, Stability analysis for the
  virtual element method, Math. Models Methods Appl. Sci. 27~(13) (2017)
  2557--2594.

\bibitem{DesiderioFallettaFerrariScuderi20221}
L.~Desiderio, S.~Falletta, M.~Ferrari, L.~Scuderi, On the coupling of the
  curved virtual element method with the one-equation boundary element method
  for 2{D} exterior {H}elmholtz problems, SIAM J. Numer. Anal. 60~(4) (2022)
  2099--2124.

\bibitem{BeiraoRussoVacca2019}
L.~Beir\~{a}o~da Veiga, A.~Russo, G.~Vacca, The virtual element method with
  curved edges, ESAIM Math. Model. Numer. Anal. 53~(2) (2019) 375--404.

\bibitem{CangianiGeorgoulisPryerSutton2017}
A.~Cangiani, E.~H. Georgoulis, T.~Pryer, O.~J. Sutton, A posteriori error
  estimates for the virtual element method, Numer. Math. 137~(4) (2017)
  857--893.

\bibitem{BrennerScott2008}
S.~Brenner, L.~Scott, The mathematical theory of finite element methods, 3rd
  Edition, Vol.~15 of Texts in Applied Mathematics, Springer, New York, 2008.

\bibitem{DupontScott1980}
T.~Dupont, R.~Scott, Polynomial approximation of functions in {S}obolev spaces,
  Math. Comp. 34~(150) (1980) 441--463.

\bibitem{MelenkSauter2022}
J.~Melenk, S.~A. Sauter,
  \href{https://arxiv.org/abs/2201.02602}{Wavenumber-explicit hp-fem analysis
  for maxwell's equations with impedance boundary conditions} (2022).
\newline\urlprefix\url{https://arxiv.org/abs/2201.02602}

\bibitem{DesiderioFallettaFerrariScuderi2022}
L.~Desiderio, S.~Falletta, M.~Ferrari, L.~Scuderi, C{VEM}-{BEM} coupling with
  decoupled orders for 2{D} exterior {P}oisson problems, J. Sci. Comput. 92~(3)
  (2022) Paper No. 96, 25.

\bibitem{OrtizBernardinAlvarezHitschfeldKahlerRussoSilvaValenzuelaOlateSanzana}
A.~Ortiz-Bernardin, C.~Alvarez, N.~Hitschfeld-Kahler, A.~Russo,
  R.~Silva-Valenzuela, E.~Olate-Sanzana, Vemlab: a matlab library for the
  virtual element method.

\bibitem{GeuzaineRemacle2009}
C.~Geuzaine, J.~Remacle, Gmsh: a three-dimensional finite element mesh
  generator with built-in pre- and post processing facilities, Internat. J.
  Numer. Methods Engrg.~(79) (2009) 1309--1331.

\bibitem{TalischiPaulinoPereiraMenezes2012}
C.~Talischi, G.~H. Paulino, A.~Pereira, I.~F.~M. Menezes, {\tt {P}oly{M}esher}:
  a general-purpose mesh generator for polygonal elements written in {M}atlab,
  Struct. Multidiscip. Optim. 45~(3) (2012) 309--328.

\bibitem{Steinbach2008}
O.~Steinbach, Numerical approximation methods for elliptic boundary value
  problems, Springer, New York, 2008, finite and boundary elements, Translated
  from the 2003 German original.

\bibitem{BrennerGuanSung2017}
S.~Brenner, Q.~Guan, L.~Sung, Some estimates for virtual element methods, Comput. Methods Appl. Math. 17~(4) (2017) 553--574.

\bibitem{DesiderioFallettaFerrariScuderi2023}
L.~Desiderio, S.~Falletta, M.~Ferrari, L.~Scuderi, {C}{VEM}-{BEM} coupling for the simulation of time-domain wave fields scattered by obstacles with complex geometries, to appear in Comput. Methods Appl. Math.

\end{thebibliography}
\end{document}